\documentclass[a4paper,12pt]{amsart}
\usepackage{amsmath,amsfonts,amssymb}
\usepackage{amsthm}
\usepackage{amstext,amscd,mathabx,mathdots}
\usepackage{graphics,graphicx,epsf}
\usepackage[usenames]{color}
\usepackage{hyperref}
\usepackage{cite}
\usepackage{tikz}
\usepackage{xcolor}
\usetikzlibrary{arrows}

\newtheorem{theorem}{Theorem}[section]

\newtheorem{proposition}[theorem]{Proposition}
\newtheorem{lemma}[theorem]{Lemma}
\newtheorem{corollary}[theorem]{Corollary}
\newtheorem{remark}[theorem]{Remark}

\renewcommand{\proof}{{\noindent \bf Proof:\ }}

\newcommand{\eproof}{\hfill \mbox{${\square}$}}

\numberwithin{equation}{section}

\title[Existence  and asymptotic behavior of solutions for a Klein-Gordon system]{Existence  and asymptotic behavior of solutions for a Klein-Gordon system}

\author[Cl\'adio O. P. Da Silva]{Cl\'adio O. P. Da Silva}
\address[C. O. P. Da Silva]{Centro de Ci\^encias Humanas e Exatas, Universidade Estadual da Para\'iba, 58500-000 Monteiro PB, Brazil.}
\email{cladio@cche.uepb.edu.br}

\author[Aldo T. Louredo]{Aldo T. Louredo}
\address[A. T. Louredo]{Departamento de Matem\'atica, Universidade Estadual da Para\'iba, 58429-500 Campina Grande PB, Brazil.}
\email{aldolouredo@gmail.com}

\author[Manuel Milla Miranda]{Manuel Milla Miranda}
\address[M. Milla Miranda]{Departamento de Matem\'atica, Universidade Estadual da Para\'iba, 58429-500 Campina Grande PB, Brazil.}
\email{millamiranda@gmail.com}

\date{\today}

\begin{document}

\maketitle

\begin{abstract}
In this paper we study the global existence and uniqueness of solution for a Klein-Gordon equations system with mixed 
boundary conditions. Also we analyze the asymptotic behavior of this solution. 

\end{abstract}


\section{Introduction} 

A mathematical model that describes the interaction of two electromagnetic fields $u$ and $v$ with masses $\sigma$ and $\varrho$, respectively, and with interaction constants $\phi>0$ and $\tau>0$ is given by following Klein-Gordon system
\begin{equation}\label{intd}
\left| \begin{array}{l}
u'' - \Delta u + \sigma^2 u+\phi v^2u =0
\\
v'' - \Delta v + \varrho^2 v + \tau u^2v =0.
\end{array}
\right.
\end{equation}
This model was proposed by Segal \cite{Segal}. Further generalizations these problem are given in Medeiros and Milla Miranda \cite{MA1}, Louredo and Milla Miranda \cite{A1} by using Galerkin methods and Medeiros and Milla Miranda \cite{MA2} by using potential well method, see also Medeiros and Menzala \cite{MM}. 

 On the non-existence of global solutions of a Klein-Gordon system we can to mention Aliev and Kazimov \cite{2015} where is considered a positive initial energy and Wang \cite{Wang}, with a non-negative potential but without damping terms in the equations. See also \cite{2016} and \cite{Liu}. The control is analyzed in Dolgopolik, Fradkov and Andrievsky \cite{2018} and in Park \cite{Park}. The asymptotic dynamical can be find in in Ferreira and Menzala \cite{FM} and Kim and Sunagawa \cite{2014}. On the decay of solutions of a Klein-Gordon-Schr\"odinger system we can to mention Almeida, Cavalcanti and Zanchetta\cite{AMZ}, Bisognin, Cavalcanti and Soriano \cite{Cavalcanti} and Yan, Boling, DaiWen and Xin \cite{2}. 
 
 In  \cite{MA1} the authors analyzed the existence and uniqueness of weak solutions of a mixed problem for system \eqref{intd} with null Dirichlet boundary conditions and coupled nonlinear terms $|v|^{\rho+2}|u|^{\rho}u$ and $|u|^{\rho+2}|v|^{\rho}v$ where $\rho$ is a real number with $\rho >-1$. They used the energy method. In \cite{MA2} the same authors studied the above problem but with coupled nonlinear terms $-|v|^{\rho+2}|u|^{\rho}u$ and $-|u|^{\rho+2}|v|^{\rho}v$. This change of sign in the coupled terms makes that the energy method does not work. They obtain the existence of weak solutions by using the potential well method which was introduced by Sattinger \cite{Sattinger}.

Let $\Omega$ be an open, bounded and connected set of $\mathbb{R}^n$ with its boundary $\Gamma$ of class $C^2$. Suppose also that $\Gamma$ is partitioned into $\Gamma_0$ and $\Gamma_1$ both with positive measure and $\overline{\Gamma}_1\cap \overline{\Gamma}_1$ empty, see Figure 1.
\begin{figure}[!h]
\begin{center}
    \includegraphics[scale=.7]{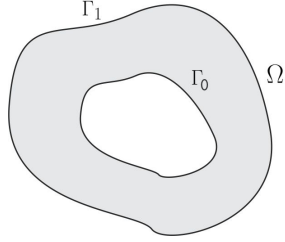} 
    \end{center}
\caption{The set $\Omega$.} 
\end{figure}

This paper is concerned with the global existence, uniqueness and uniform stability of the energy for the following initial-boundary value problem.
\begin{equation}\label{int}
\quad \quad \left| \begin{array}{lcc}
u'' - \Delta u + |u|^{\rho}|v|^{\rho}v =0; & \mbox{in} \quad &\Omega \times (0,\infty)\\
v'' - \Delta v + |u|^{\rho}u|v|^{\rho}=0; & \mbox{in} \quad &\Omega \times (0,\infty)\\
u=v=0; & \mbox{on} \quad &\Gamma_0 \times (0,\infty)\\
\frac{\partial u}{\partial \nu} + \delta u'=0 & \mbox{on} \quad &\Gamma_1 \times (0,\infty)\\
\frac{\partial v}{\partial \nu} + \delta v'=0 & \mbox{on} \quad &\Gamma_1 \times (0,\infty) \\
u(0)=u^0,v(0)=v^0, u'(0)=u^1,v'(0)=v^1, 
\end{array}
\right.
\end{equation}
where $\rho$ is a positive real number and $\nu(x)$ denotes the unit outward normal vector at $x \in \Gamma_1$ and $\delta$ is a real function belong to $W^{1,\infty}(\Gamma_1)$ such that $\delta(x)\geqslant \delta_0>0$ on $\Gamma_1$.
 
 We note that the energy of system \eqref{int} given by
\begin{equation*}
\begin{split}
E(t)&=\frac{1}{2}\{|u'(t)|^2 + |v'(t)|^2\} + \frac{1}{2}\{\|u(t)\|^2 + \|v(t)\|^2 \} \\
&+ \frac{1}{\rho +1}\int_{\Omega} (|u(t)|^{\rho}u(t))(|v(t)|^{\rho}v(t))dx,~ t \in [0, \infty).
\end{split}
\end{equation*}
does not definite sign. Therefore the energy method to obtain global solution of \eqref{int} does not work. To overcome this serious difficulty we use a method introduced by MiIla Miranda, Louredo and Medeiros \cite{Carrier}, which was inspirated in one idea of  Tartar \cite{LT}.  This method simplifies the potential well one. We complement our approach by using the Galerkin method with a special basis, due to the dissipative boundary conditions, and compactness argument. With the above considerations, we obtain a global weak solution of \eqref{int} with restrictions on the norm of initial data and on $\rho >0$ which depends on the embedding of the Sobolev spaces. 

The uniqueness of solutions is derived by using the energy method. Thus if $\rho \geqslant 1$, we consider $n=1,2$ and  if $\rho = 1$ we consider $n = 3$. This restriction on $n$ is due that is this part we need differentiate with respect to $t$ the difference of the nonlinear parts in order to apply the Mean Value Theorem.

To obtain the decay of the energy of problem \eqref{int}, we consider the same restrictions of the uniqueness of solutions and make $\delta(x)=m(x) \cdot \nu(x)$, where $m(x)=x-x^0, ~ x,x^0 \in \mathbb{R}^n$. In this conditions, by using the multiplier method and the ideas contained in Komornik and Zuazua \cite{Zuazua} and Komornik \cite{Komornik}, we obtain the exponential decay of the energy.   
  
This paper is organized as follow. In Section 2, we present some notations, hypotheses and main results and also we show results related to density, separability and trace. In Section 3 we established the proof of the existence of global solution and uniqueness. In Section 4 we prove the result about exponential decay of the energy.
\section{Notations and Main Result}
In this section, we present the notations, assumptions and the main results, however the proof will be developed in the next sections. Also we show results related to density, separability and trace that will be important throughout this paper. We start with some notations and hypotheses.
\subsection{Notation and hypotheses}
 The inner product and norm of $L^2(\Omega)$ are represented, respectively, by $(\cdot,\cdot)$ and $|\cdot|$. By $V$ is denoted the Hilbert space 
\[
V=\{u \in H^{1}(\Omega) \ ; \ u=0 \quad \mbox{on} \quad \Gamma_0\}
\]
equipped with the inner product and norm, respectively,
\[
((u,v))= \sum_{i=1}^{n}\int_{\Omega}\frac{\partial u}{\partial x_i}\frac{\partial v}{\partial x_i} dx \quad \mbox{and} \quad \|u\|^2= \sum_{i=1}^{n} \int_{\Omega}\left|\frac{\partial u}{\partial x_i}\right|^{2}.
\]

Let $\theta$ be a real number with $1\leqslant \theta<2$ such that $\frac{1}{\theta}+\frac{1}{\theta'}=1$. We consider the following Banach spaces equipped with respective norms
\[
W^{1,\theta'}(\Omega), \quad \|u\|= \left(\int_{\Omega}|u(x)|^{\theta'}dx+\sum_{i=1}^{n} \int_{\Omega}\left|\frac{\partial u(x)}{\partial x_i}\right|^{\theta'}dx\right)^{\frac{1}{\theta'}}
\]
and 
\[
W_{\Gamma_0}^{1,\theta'}(\Omega)=\{u\in W^{1,\theta'}(\Omega); u=0 ~ \mbox{on} ~ \Gamma_0\}; \quad \|u\|= \left(\sum_{i=1}^{n} \int_{\Omega}\left|\frac{\partial u(x)}{\partial x_i}\right|^{\theta'}dx\right)^{\frac{1}{\theta'}}.
\]

We also consider o Banach space
\[
\mathcal{X}=\{u \in V ; \Delta u \in L^{\theta}(\Omega)\}
\]
with the norm
\[
\|u\|_{\mathcal{X}}=\|u\|_V+\|\Delta u\|_{L^{\theta}(\Omega)}.
\]

Introduce the hypotheses
\begin{equation}\label{n=1,2}
\rho > 0 ~ \mbox{and} ~ \theta>1 ~\mbox{with} ~ 4\rho\theta\geqslant 1, \quad  \mbox{if} \quad n=1,2;  
\end{equation}
\begin{equation}\label{3n6}
\dfrac{n+2}{8n} \leqslant \rho \leqslant  \dfrac{n+2}{4(n-2)},  \quad \mbox{if} \quad 3 \leqslant n\leqslant 6.
\end{equation}
\begin{remark}\label{obsH1}
	$(i)$ Note that for $n \geqslant 3$ we have $0<\rho< \dfrac{2}{n-2}$, then $1<2(\rho+1)\leqslant q$ therefore the following embeddings of Sobolev
	\[
	V \hookrightarrow L^{q}(\Omega) \hookrightarrow L^{2(\rho+1)}(\Omega),
	\]
	are holds, where $q=\frac{2n}{n-2}$. In particular for $\rho=1$ we have $V \hookrightarrow L^{4}(\Omega)$. Thus, there exists positive constsnts $c_0$ and $c_1$ such that
	\begin{equation}\label{H1}
		\|v\|_{L^{2(\rho+1)}(\Omega)} \leqslant c_0||v||, \quad \mbox{and} \quad \|v\|_{L^{4}(\Omega)} \leqslant c_1||v|| ~\forall ~ v \in  V.
	\end{equation}
	\\
	$(ii)$ Under the restrictions $\eqref{3n6}$ on $\rho$ and $n$ we obtain
	\[
	V\hookrightarrow L^{q}(\Omega)\hookrightarrow L^{\frac{8n\rho}{n+2}}(\Omega) \quad \mbox{and} \quad V\hookrightarrow L^{q}(\Omega)\hookrightarrow L^{\frac{4n}{n+2}}(\Omega).
	\]
\end{remark}	

Introduce the following restrictions on the initial data and some constants:
\begin{equation}\label{restriction}
\|u_0\|, ~\|v_0\| < \lambda^{*} \quad \text{and} \quad L < {\frac14} (\lambda^{*})^2,
\end{equation}
where
\begin{equation}\label{Hc21}
	\lambda^{*} =\left( \dfrac{1}{4N}\right)^{\frac{1}{2\rho}};
\end{equation}

\begin{equation}\label{Hc22}
	L=\dfrac{1}{2}\left[|u^1|^2 + |v^1|^2\right]+ \dfrac{1}{2}\left[\|u^0\|^2 + \|v^0\|^2\right]+ N\left[\|u^0\|^{2(\rho+1)} + \|v^0\|^{2(\rho+1)}\right];
\end{equation}

\begin{equation}\label{Nc}
	N=\frac{c_0^{2(\rho +1)}}{2(\rho +1)} ;
\end{equation}

\begin{equation}\label{delta}
	\delta \in W^{1,\infty}(\Gamma_1) \quad \mbox{such that} \quad \delta(x)\geqslant \delta_0>0 \quad \mbox{on} \quad \Gamma_1.
\end{equation}

\begin{theorem}\label{theorem1}
	Assume hypotheses \eqref{3n6} and \eqref{restriction}--\eqref{delta}. Consider $u^0, v^0 \in V $ and $u^1, v^1 \in L^2(\Omega)$. Then there exists functions $u,v $ in the class
	\begin{equation*}\label{class1}
	\left| \begin{array}{l}
	u, v \in L^{\infty}(0,\infty; V),  
	\\
	u', v' \in L^{\infty}(0,\infty; L^2(\Omega))
	\\
	u', v' \in L^{\infty}(0,\infty; L^2(\Gamma_1)), 

	\end{array}
	\right.
	\end{equation*}
	such that $u$ and $v$ satisfies the equations
	\begin{equation*}\label{eq1}
	\begin{split}
	&u'' - \Delta u + |u|^{\rho}|v|^{\rho}v =0 \quad \text{in} \quad H^{-1}_{loc}(0,\infty;L^{q'}(\Omega))\\ 
	&v'' - \Delta v + |u|^{\rho}u|v|^{\rho} =0 \quad \text{in} \quad  H^{-1}_{loc}(0,\infty; L^{q'}(\Omega))\\
	&\frac{\partial u}{\partial \nu} + \delta(x)u'=0 \quad \mbox{in} \quad  L_{loc}^{2}(0,\infty; L^{2}(\Gamma_1))\\
	&\frac{\partial v}{\partial \nu} + \delta(x)v'=0 \quad \mbox{in} \quad L_{loc}^{2}(0,\infty; L^{2}(\Gamma_1)).
	\end{split}
	\end{equation*}
	and the initial conditions
	\[
	u(0)=u^0, v(0)=v^0 \quad u'(0)=u^1 ,~v'(0)=v^1
	\]
\end{theorem}
\begin{corollary}
We obtain similar results to Theorem \ref{theorem1} for the case $\rho>0$ and $n=1,2$.
\end{corollary}

Set the hypothesis
\begin{equation}\label{n=1,2s}
\rho > 0 \quad \mbox{and} \quad \theta>1 \quad \mbox{with} \quad 4\rho\theta\geqslant 1, \quad  \mbox{if} \quad n=1,2; 
\end{equation}
\begin{equation}\label{7n11}
\rho =  \dfrac{2}{n-2} \quad \mbox{and} \quad \theta=\dfrac{n}{n-2},  \quad \mbox{if} \quad 7 \leqslant n\leqslant 11.
\end{equation}

\begin{remark}
	Under the restrictions \eqref{7n11} on $\rho$ and $n$ we have:   
	\[
	V \hookrightarrow L^{q}(\Omega) \hookrightarrow L^{4\rho \theta}(\Omega), \quad \mbox{and} \quad  V \hookrightarrow  L^{2\theta}(\Omega)
	\]
\end{remark}	

\begin{theorem}\label{theorem2}
	Consider $u^0, v^0 \in V $ and $u^1, v^1 \in L^2(\Omega)$. Then under hypothesis \eqref{restriction}--\eqref{delta} and \eqref{7n11}, we have that there exists functions $u,v $ in the class
	\begin{equation*}\label{class2}
	\left| \begin{array}{l}
	u, v \in L^{\infty}(0,\infty; V);
	\\
	u', v' \in L^{\infty}(0,\infty; L^2(\Omega));
	\\
	u', v' \in L^{\infty}(0,\infty; L^2(\Gamma_1)).
	\end{array}
	\right.
	\end{equation*}
	such that $u$ and $v$ satisfies the equations
	\begin{equation*}\label{eq2}
	\begin{split}
	&u'' - \Delta u + |u|^{\rho}|v|^{\rho}v =0 \quad \text{in} \quad H^{-1}_{loc}(0,\infty;L^{\theta}(\Omega))\\ 
	&v'' - \Delta v + |u|^{\rho}u|v|^{\rho} =0 \quad \text{in} \quad  H_{loc}^{-1}(0,\infty; L^{\theta}(\Omega))\\
	&\frac{\partial u}{\partial \nu} + \delta(x)u'=0 \quad \mbox{in} \quad L^{2}_{loc}(0,\infty; L^{2}(\Gamma_1))\\
	&\frac{\partial v}{\partial \nu} + \delta(x)v'=0 \quad \mbox{in} \quad L^{2}_{loc}(0,\infty; L^{2}(\Gamma_1)).
	\end{split}
	\end{equation*}
	and the initial conditions
	\[
	u(0)=u^0, v(0)=v^0 \quad u'(0)=u^1 ,~v'(0)=v^1
	\]
\end{theorem}
\begin{corollary}
	We obtain similar results to Theorem \ref{theorem2} for the case $\rho>0$ and $n=1,2$.
\end{corollary}

In order to obtain results on the uniqueness and decay of solutions of Problem \ref{int}, we prove the following theorem of the existence of solutions for the case $\rho=1$ and $n=1,2,3$. For the case $\rho>1$ and $n>3$, the regularity of the solution $u$ obtained below is an open problem.

\begin{remark}
	We observe that for $0 < \rho \leqslant \frac{1}{n-2}$ and from trace theorem, we have
	\[
	V \hookrightarrow H^{\frac{1}{2}}(\Gamma_1)\hookrightarrow L^{q_1}(\Gamma_1) \hookrightarrow L^{2(\rho+1)}(\Gamma_1),
	\]
	where $q_1=\frac{2(n-1)}{n-2}$ for $n\geqslant 3$. In particular for $\rho =1$ and $n=3$, we obtain
	\[
	V \hookrightarrow H^{\frac{1}{2}}(\Gamma_1)\hookrightarrow L^{4}(\Gamma_1) \hookrightarrow L^{2}(\Gamma_1).
	\]
	Thus there exists positive constants $c_2$ and $c_3$ such that
	\begin{equation}\label{gamma}
	\|w\|_{L^{4}(\Gamma_1)} \leqslant c_2 \|w\|, \quad \mbox{and} \quad \|w\|_{L^{2}(\Gamma_1)} \leqslant c_3 \|w\|,~ \forall w \in V .
	\end{equation}
\end{remark}  
We also consider the following hypothesis: 
\\
Introduce the following restrictions on the initial data and some constants:
\begin{equation}\label{rest}
\|u_0\|, ~\|v_0\| < \lambda_1^{*} \quad \text{and} \quad L _1< {\frac14} (\lambda_1^{*})^2,
\end{equation}
where
\begin{equation}\label{H21}
	\lambda_1^{*} =\left( \dfrac{1}{4N_1}\right)^{\frac{1}{2}};
\end{equation}

\begin{equation}\label{H22}
	L_1=\dfrac{1}{2}\left[|u_1|^2 + |v_1|^2\right]+ \dfrac{1}{2}\left[\|u_0\|^2 + \|v_0\|^2\right]+ N_1\left[\|u_0\|^{4} + \|v_0\|^{4}\right];
\end{equation}

\begin{equation}\label{N}
	N_1=\frac{c_1^{4}}{2} \left[n+\frac{1}{4}\right]+ \frac{Rc_2^4}{2}+c_1^{4}(n-1);
\end{equation}


\begin{theorem}\label{theorem3}
	Let $\rho=1$ and $n\leqslant3$. Consider \eqref{rest}--\eqref{N} and $u^0, v^0 \in V \cap H^2(\Omega)$ and $u^1, v^1 \in V$ satisfying
	\[
	\begin{split}
	&\frac{\partial u^0}{\partial \nu} + \delta(x)u^1=0 \quad \mbox{on} \quad \Gamma_1 \\
	&\frac{\partial v^0}{\partial \nu} + \delta(x)v^1=0 \quad \mbox{on} \quad \Gamma_1.
	\end{split}
	\]
	Then there exists unique functions $u,v $ in the class
	\begin{equation}\label{class3}
	\left| \begin{array}{l}
	u, v \in L^{\infty}(0,\infty; V\cap H^2(\Omega)), ~ u', v' \in L^{\infty}_{loc}(0,\infty; V)
	\\
	u'', v'' \in L_{loc}^{\infty}(0,\infty; L^2(\Omega))
	\\
	u', v' \in L^{\infty}(0,\infty; L^2(\Gamma_1)), \quad u'', v'' \in L_{loc}^{\infty}(0,\infty; L^2(\Gamma_1)) 
	\end{array}
	\right.
	\end{equation}
	such that $u$ and $v$ satisfies the equations
	\begin{equation*}\label{eq3}
	\begin{split}
	&u'' - \Delta u + |u||v|v =0 \quad \text{in} \quad  L^{\infty}_{loc}(0,\infty;L^{2}(\Omega))\\ 
	&v'' - \Delta v + |u|u|v| =0 \quad \text{in} \quad L_{loc}^{\infty}(0,\infty; L^{2}(\Omega))\\
	&\frac{\partial u}{\partial \nu} + \delta(x)u'=0 \quad \mbox{in} \quad L^{\infty}_{loc}(0,\infty; H^{\frac{1}{2}}(\Gamma_1))	\\
	&\frac{\partial v}{\partial \nu} + \delta(x)v'=0 \quad \mbox{in} \quad  L^{\infty}_{loc}(0,\infty; H^{\frac{1}{2}}(\Gamma_1)).
	\end{split}
	\end{equation*}
	and the initial conditions
	\[
	u(0)=u^0, v(0)=v^0 \quad u'(0)=u^1 ,~v'(0)=v^1
	\]
\end{theorem}

\begin{corollary}\label{corollary3}
	We obtain similar results to Theorem \ref{theorem3} for the case $\rho>1$ and $n=1,2$.
\end{corollary}

Next we state the result on the decay of solutions of Theorem \ref{theorem3}. We assume that there exists a point $x^0 \in \mathbb{R}^n$, such that
\[
\Gamma_0=\{x\in \Gamma ; m(x) \cdot \nu(x) \leqslant 0\} \quad \mbox{and} \quad \Gamma_1=\{x\in \Gamma ; m(x) \cdot \nu(x) > 
0\}.
\]
where $m(x)=x-x^0, ~ x\in\mathbb{R}^n$, see Figure 2
\begin{figure}[!h]
\begin{center}
    \includegraphics[scale=.7]{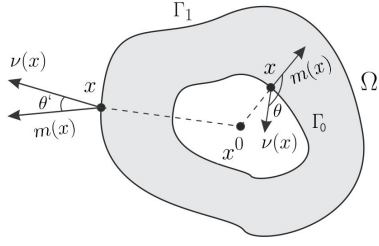}
\end{center}
\caption{The sets $\Gamma_0$ and $\Gamma_1$.} 
\end{figure}

In this section we consider $\delta(x)=m(x)\cdot\nu(x)$ and $R=\displaystyle\max_{x \in \overline{\Omega}} \|m(x)\|$.
The energy of system $\eqref{int}$ with $\rho=1$ is given by
\begin{equation*}\label{E}
E(t)=\frac{1}{2}\{|u'(t)|^2 + |v'(t)|^2  + \|u(t)\|^2 + \|v(t)\|^2 \} 
+ \frac{1}{2}\int_{\Omega} (|u(t)|u(t))(|v(t)|v(t))dx,~ t \in [0, \infty).
\end{equation*}

We have the following result:
\begin{theorem}\label{teo2.2}
	Let $\{u,v\}$ be the solution obtained in Theorem \ref{theorem3}. Then 
	\begin{equation}\label{EE22}
	E(t) \leqslant 3E(0)e^{-\frac{\tau}{3} t}, \quad \forall ~t \in[0,\infty)
	\end{equation}
	where
	\begin{equation}\label{const}
	\begin{split}
	&\tau=\min\left\{\frac{1}{2P},\frac{m_0}{D}\right\}>0; \\
	& P=4\left(2R+ \frac{n-1}{2}+\frac{n-1}{2\lambda_1}\right); \\
	&D=R^3+R+R^2(n-1)^2c_3^2;\\
	& m_0=\min\{m(x) \cdot \nu(x) ; x \in \Gamma_1\}>0.
	\end{split}
	\end{equation}
\end{theorem}
\subsection{Separability}
Now consider $X, Y$ and $W$ be theree Banach spaces such that $W\hookrightarrow X$ and $W\hookrightarrow Y$. Let $Z$ be a topological vector space that separates points, such that $X\hookrightarrow Z$ and $Y\hookrightarrow Z$. Then the space $E=X \cap Y$ provided with the norm
\[
\|u\|_{E}=\|u\|_{X}+\|u\|_{Y}
\]
is a Banach space.
\begin{proposition}\label{p1}
	If $W$ is dense in $X$ and dense in $Y$ then $W$ is dense $E$.	
\end{proposition}	
\proof
	Consider $T \in E'$ such that
	\[
	\langle T, w \rangle_{E'\times E}=0, \quad \forall w \in W.
	\]
	Note that $W$ has the same topology considered as a subspace of $X \cap Y$ or as a subspace of $X \times Y$. So $T$ is continuous on $W$ with the topology of $X \cap Y$. Then by the Hanah-Banach Theorem there exist $R \in X'$ and $S \in Y'$ such that 
	\begin{equation}\label{T1}
		\langle T, w \rangle_{E'\times E}=	\langle R, w \rangle_{X'\times X}+	\langle S, w \rangle_{Y'\times Y}, \quad \forall w \in W.
	\end{equation}
	Observe that $X'\hookrightarrow W'$ and $Y'\hookrightarrow W'$. Then 
	\begin{equation}\label{RS}
		\begin{split}
			&\langle R, w \rangle_{X'\times X}=\langle R, w \rangle_{W'\times W}	\\
			&\langle S, w \rangle_{Y'\times Y}=\langle S, w \rangle_{W'\times W}.
		\end{split}
	\end{equation}
	By \eqref{T1} and \eqref{RS}, we obtain	
	\begin{equation}\label{Tw}
		\langle R+S, w \rangle_{W'\times W}= \langle T, w \rangle_{E'\times E}	=0, \quad \forall w \in W.
	\end{equation}
	By the density of $W$ in $X$ and in $Y$, \eqref{Tw} implies
	\[
	R+S=0 ~ \mbox{on} ~ X  \quad \mbox{and} \quad R+S=0 ~ \mbox{on} ~ Y.
	\]
	Therefore 
	\[
	T=R+S=0 ~ \mbox{on} ~ E . 
	\]
\eproof	

\begin{proposition}\label{p2}
	Assume the hypotheses of Proposition \ref{p1}. Then if $W$ is separable we have that $E$ is separable.
\end{proposition}	
\proof
	Let $\{w_1,w_2, \ldots \}$ be a basis of $W$. Consider $u\in E$. Then by Proposition \ref{p1}, there exists $\varphi \in W$ such that
	\[
	\|u-\varphi\|_{E}< \frac{\epsilon}{2}.
	\] 
	Also there exists $\displaystyle\sum_{j=1}^{n}\alpha_j w_j$ such that 
	\[
	\left\|\varphi - \sum_{j=1}^{n}\alpha_j w_j \right\|_{W}< \frac{\epsilon}{2c}.
	\]	
	Thus
	\[
	\left\|u-\sum_{j=1}^{n}\alpha_j w_j\right\|_{E} \leqslant \|u-\varphi\|_{E} + \left\|\varphi - \sum_{j=1}^{n}\alpha_j w_j \right\|_{E} \leqslant  \|u-\varphi \|_{E}  + c \left\|\varphi - \sum_{j=1}^{n}\alpha_j w_j\right\|_{W} < \epsilon.
	\] 
\eproof
\subsection{A trace theorem}

In what follows we will show that $\mathcal{D}(\overline\Omega)$ is dense in $\mathcal{X}$. For this, let $\mathcal{O}$ be a star-shaped subset of $\mathbb{R}^n$ with respect to $0 \in\mathbb{R}^n$. Consider the linear homotetic transformation $\sigma_{\eta}(x)=\eta x, ~ \eta>0$. Note that for $\eta>1$,
\begin{equation}\label{O}
	\mathcal{O} \subset \overline{\mathcal{O}} \subset \sigma_{\eta}(\mathcal{O}).
\end{equation}
Consider a function $w:\mathcal{O} \to \mathbb{R}$ defined in $\mathcal{O}$. For $\eta>0$ introduce the function
\[
\sigma_{\eta} \circ w:  \sigma_{\eta}(\mathcal{O}) \to \mathbb{R}, \quad y \longmapsto (\sigma_{\eta} \circ w)(y)= w(\sigma_{\frac{1}{\eta}}(y)).
\]
Note that when $\eta>1$, the domain of the function $\sigma_{\eta} \circ w$ contain the domain of $w$ (see \eqref{O})
\begin{proposition}\label{Temam}
	Let $S\in \mathcal{D}'(\mathcal{O})$. Then \\
	\\
	$1)$ $\sigma_{\eta} \circ S$ defined by
	\[
	\langle \sigma_{\eta} \circ S, \xi \rangle = \frac{1}{\eta^n} \langle S, \sigma_{\eta} \circ \xi \rangle , \quad \xi \in \mathcal{D}(\mathcal{\sigma_{\eta}(\mathcal{O})}),
	\]
	belongs to $\mathcal{D}'(\sigma_{\eta}(\mathcal{O})), ~ (\eta>0)$. \\
	\\
	$2)$ $\dfrac{\partial}{\partial y_i}(\sigma_{\eta} \circ S)= \eta \sigma_{\eta} \circ \left(\dfrac{\partial}{\partial y_i} S \right) , ~ (\eta>0)$. \\
	\\
	$3)$ If $\eta >1, ~\eta \to1$, the restriction to $\mathcal{O}$ of $\sigma_{\eta} \circ S$  convergs to $S$ in the distribution sense. \\
	\\
	$4)$ If $v \in L^p(\mathcal{O}), ~(1\leqslant p < \infty), ~ \sigma_{\eta} \circ v \in L^p(\sigma_{\eta}(\mathcal{O})), ~(\eta>0)$. For $\eta>1, ~\eta \to 1$, the restriction to $\mathcal{O}$ of $\sigma_{\eta} \circ v$ convergs to $v$ in $L^p((\mathcal{O}))$.
\end{proposition}	
\proof
	The proof can be found in Temam \cite[Lemma 1.1, p. 7]{RT}.
\eproof

We have the following results:
\begin{theorem}
	The space $\mathcal{D}(\overline{\Omega})$ is dense in $\mathcal{X}$.
\end{theorem}
\proof
	Let $U$ be an open set of $\mathbb{R}^n$ with boundary $\partial U$ of class $C^2$. Introduce the Banach space
	\[
	\mathcal{X}(U)=\{u \in V(U) ; \Delta u \in L^{\theta}(U)\}
	\]	
	equipped with the norm
	\[
	\|u\|_{\mathcal{X}(U)}=\|u\|_{V(U)} + \|\Delta u\|_{L^{\theta}(U)}.
	\]
	We divide the proof in four parts.
	\\
	{\bf First part.} By truncation and regularization we prove that $\mathcal{D}(\mathbb{R}^n)$ is dense in $\mathcal{X}(\mathbb{R}^n)$. For more details see Medeiros and Milla Miranda \cite[Theorem 1.1, p. 8]{MM}.
	\\
	{\bf Second part.} Let $(U_k)_{1\leqslant k \leqslant m}$ be an open covering of $\Gamma_0$ and $\Gamma_1$ with $U_k^+=\Omega \cap U_k$ star-shaped with respect to one of its points, $k=1, \ldots,m$. Let $(\varphi_k)_{0\leqslant k \leqslant m}$ be a $C^{\infty}-$partition of unity subordinate to the open covering $\Omega, (U_k)_{1\leqslant k \leqslant m}$ of $\overline{\Omega}$. Thus
	\[
	\varphi_0(x)+\sum_{k=1}^{m} \varphi_{k}(x)=1, \quad \forall x\in \overline{\Omega}, ~ \varphi_0 \in \mathcal{D}(\Omega), ~ \varphi_k \in \mathcal{D}(U_k), k=1, \ldots,m.
	\]
	Consider $u\in \mathcal{X}$. Then
	\begin{equation}\label{phi}
		u=\varphi_0(x)u+\sum_{k=1}^{m} \varphi_{k}(x)u.
	\end{equation}
	We use the notations
	\[
	v_k=\varphi_k(x) u, \quad k=0,1, \ldots,m.
	\]
	We will analize $v_0$. Represent by $U_0$ an open set of $\mathbb{R}^n$ such that $(\operatorname{supp} \varphi_0) \cap \Omega$ is contained in $U_0$. After translation, we can choose $U_0$ such that $U_0$ is star-shaped with respect to $0 \in \mathbb{R}^n$. Define $\sigma_{\eta} \circ v_0, ~ \eta>1$. Then by \eqref{O} and Proposition \ref{Temam}, first part, we have that $\sigma_{\eta} \circ v_0$ is defined in $\sigma_{\eta}(U_0)$. Consider
	\[
	\psi \in \mathcal{D}(\sigma_{\eta}(U_0)) ~ \mbox{such that} ~ \psi \equiv 1 ~\mbox{on} ~U_0, \quad \mbox{and} \quad w_{0\eta}=\psi[\sigma_{\eta} \circ v_0], ~ \eta >1.
	\]
	Then $\operatorname{supp}(w_{0\eta})$ is contained in $\sigma_{\eta}(U_0)$. By Proposition \ref{Temam}, second part, we obtain
	\begin{equation}\label{01}
		\frac{\partial w_{0\eta}}{\partial x_i}= \frac{\partial \psi}{\partial x_i}[\sigma_{\eta} \circ v_0]+ \eta \psi \left(\sigma_{\eta}\circ \frac{\partial v_0}{\partial x_i}\right)
	\end{equation}
	
	\begin{equation}\label{02}
		\Delta w_{0\eta}= \eta^2\psi[\sigma_{\eta} \circ \Delta v_0]+ \Delta \psi [\sigma_{\eta} \circ \Delta v_0] + 2\eta \sum_{i=n}^{n}\frac{\partial \psi}{\partial x_i} \left[\sigma_{\eta}\circ \frac{\partial v_0}{\partial x_i}\right].
	\end{equation}
	
	By the preceding equalities, we obtain that $w_{0\eta} \in \mathcal{X}(\sigma_{\eta}(U_0))$. Consider $\tilde{w}_{0\eta}$ the extension of $w_{0\eta}$, that is,
	\begin{equation*}
		\tilde{w}_{0\eta}=\left| \begin{array}{l}
			w_{0\eta} \quad \mbox{in} \quad \sigma_{\eta}(U_0);
			\\ \\
			0  \quad \mbox{in} \quad \mathbb{R}^n/\sigma_{\eta}(U_0).
		\end{array}
		\right.
	\end{equation*}
	Then $\tilde{w}_{0\eta} \in \mathcal{X}(\mathbb{R}^n)$. By the first part we have $\tilde{w}_{0\eta}$ can be approximated in $\mathcal{X}(\mathbb{R}^n)$ by functions of $\mathcal{D}(\mathbb{R}^n)$. Consequentely
	\begin{equation}\label{nome}
		w_{0\eta} ~ \mbox{can be approximated in} ~ \mathcal{X}(\sigma_{\eta}(U_0)) ~ \mbox{by functions of} ~ \mathcal{D}(\overline{\sigma_{\eta}(U_0)}).
	\end{equation}
	
	By \eqref{01} and \eqref{02} we have
	\[
	w_{0\eta}|_{U_0}={[\sigma_{\eta} \circ v_0]|}_{U_0}, \quad 
	{ \frac{\partial w_{0\eta}}{\partial x_i}|}_{U_0}= \eta {\left[\sigma_{\eta}\circ \frac{\partial v_0}{\partial x_i}\right]|}_{U_0}, \quad {\Delta w_{0\eta}|}_{U_0}=\eta^2{[\sigma_{\eta}\circ \Delta v_0]|}_{U_0}.
	\]
	Then by Proposition \ref{Temam}, third and fourth part, we obtain
	\[
	\begin{split}
	& w_{0\eta}|_{U_0} \to v_0 \quad  \mbox{in} \quad L^2(U_0) \quad \mbox{as} \quad \eta \to 1;\\
	& { \frac{\partial w_{0\eta}}{\partial x_i}|}_{U_0} \to \frac{\partial v_0}{\partial x_i} \quad \mbox{in} \quad L^2(U_0) \quad \mbox{as} \quad \eta \to 1;\\
	& {\Delta w_{0\eta}|}_{U_0} \to \Delta v_0 \quad  \mbox{in} \quad L^{\theta}(U_0) \quad \mbox{as} \quad \eta \to 1.
	\end{split}
	\]
	By \eqref{nome} and the last three convergences we conclude that $v_0$ can be approximated in $\mathcal{X}(U_0)$ by functions of $\mathcal{D}(\overline{U}_0)$.
	\\
	{\bf Third part.} Analize $v_k, ~ k=1, \ldots,m$. In this case we apply similar arguments to those used in the case $v_0$. Thus we take $U_k^+$ instead $U_0$. We can assume that $U_k^+$ is star-shaped with respect to $0 \in \mathbb{R}^n$. Consider $\sigma_{\eta}(U_k^+)$ instead $\sigma_{\eta}(U_0)$. Introduce
	\[
	\psi \in \mathcal{D}(\sigma_{\eta}(U_k^+)) \quad \mbox{with} \quad \psi \equiv 1 \quad \mbox{on}\quad U_k^+.
	\]
	Consider $w_{k\eta}=\psi[\sigma_{\eta} \circ v_k], ~ \eta>1$. Then
	\[
	w_{k\eta} \in \mathcal{X}(\sigma_{\eta}(U_k^+)); \quad \operatorname{supp}(w_{k\eta}) \subset \mathcal{X}(\sigma_{\eta}(U_k^+)); \quad \tilde{w}_{k\eta} \in  \mathcal{X}(\mathbb{R}^n)
	\]
	and
	\[
	{w_{k\eta}|}_{U_k^+} \to v_k \quad \mbox{in} \quad \mathcal{X}(\sigma_{\eta}(U_k^+)) \quad \mbox{as} \quad \eta \to 1.
	\]
	Thus $v_k$ can be approximated in $\mathcal{X}(U_k^+)$ by functions of $\mathcal{D}(\overline{U_k^+})$. 
	
	By \eqref{phi} and the above results we conclude that $u \in \mathcal{X}$ can be approximated in $\mathcal{X}(U)$ by functions of $\mathcal{D}(\overline{U})$.
	\\ 
	{\bf Fourth part.} The theorem follow since that $\mathcal{X}(U)$ and $\mathcal{X}$ has equivalent norms in $\mathcal{X}$.
\eproof


It is known by trace theorem that there exists a linear continuous and sobrejetive map 
\[
\gamma_0 : W^{1,\theta'}(\Omega) \to W^{\frac{1}{\theta},\theta'}(\Gamma), \quad  \gamma_0u= u|_{\Gamma} 
\]
and has inverse continuous, 
\[
W^{\frac{1}{\theta},\theta'}(\Gamma) \to W^{1,\theta'}(\Omega), \quad  \xi \longmapsto u
\]
In particular we have
\[
\gamma_0 : W_{\Gamma_0}^{1,\theta'}(\Omega) \to W^{\frac{1}{\theta},\theta'}(\Gamma_1), \quad  \gamma_0u= u|_{\Gamma_1} 
\]
and
\[
W^{\frac{1}{\theta},\theta'}(\Gamma_1) \to W_{\Gamma_0}^{1,\theta'}(\Omega), \quad  \xi \longmapsto u,
\]
are continuous. For more details see \cite[Theorem 5.5, p. 95]{necas}.

We want to prove a similar result for the functions in $\mathcal{X}$. We have the following trace theorem, which means that we can define $\frac{\partial u}{\partial \nu}$ on $\Gamma_1$ when $u\in \mathcal{X}$.

\begin{theorem}\label{tracoX}
	There exists a linear continuous map 
	\[
	 \mathcal{X} \to W^{-\frac{1}{\theta},\theta}(\Gamma_1), \quad u \longmapsto \gamma_1u = \frac{\partial u}{\partial \nu}
	 \]
	such that 
	\begin{equation}\label{green}
		\langle \gamma_1 u, \gamma_0 z\rangle_{W^{-\frac{1}{\theta},\theta}(\Gamma_1) \times W^{\frac{1}{\theta}, \theta'}(\Gamma_1)}=\langle \Delta u, z\rangle_{L^{\theta}(\Omega) \times L^{\theta'}(\Omega)} + \sum_{i=1}^{n} \int_{\Omega}\frac{\partial u}{\partial x_i} \frac{\partial z}{\partial x_i} dx,
	\end{equation}
	for all $z \in W_{\Gamma_0}^{1,\theta'}(\Omega)$.	
\end{theorem}	
\proof
	Note that \eqref{green} is well defined and hold for $u\in \mathcal{D}(\overline{\Omega})$. In fact, let $u\in \mathcal{D}(\overline{\Omega})$, using $W_{\Gamma_0}^{1,\theta'}(\Omega) \hookrightarrow V$ and $W_{\Gamma_0}^{1,\theta'}(\Omega) \hookrightarrow L^{\theta'}(\Omega)$, we have
\[
	\begin{split}
	\left|\int_{\Gamma_1}(\gamma_1u)(\gamma_0z) d\Gamma \right| & \leqslant \|u\|_{V} \|z\|_{V} +   \|\Delta u\|_{L^{\theta}(\Omega)} \|z\|_{L^{\theta'}(\Omega)} \\
	& \leqslant c \|u\|_{V} \|z\|_{W_{\Gamma_0}^{1,\theta'}(\Omega)} + c  \|\Delta u\|_{L^{\theta}(\Omega)} \|z\|_{W_{\Gamma_0}^{1,\theta'}(\Omega)}\\
	& \leqslant c \|u\|_{\mathcal{X}} \|z\|_{W_{\Gamma_0}^{1,\theta'}(\Omega)},
	\end{split}
	\]
	for some positive constant $c$.
	
	Let $\xi \in W^{\frac{1}{\theta},\theta'}(\Gamma_1)$. Then by trace theorem  there exists $z \in W_{\Gamma_0}^{1,\theta'}(\Omega)$ such that $\xi =\gamma_0 z$ and 
\[
\|z\|_{W_{\Gamma_0}^{1,\theta'}(\Omega)} \leqslant c \|\xi\|_{W^{\frac{1}{\theta},\theta'}(\Gamma_1)}.
\]	
Thus 
\[
\left|\int_{\Gamma_1}(\gamma_1u)\xi d\Gamma \right|  \leqslant  c \|u\|_{\mathcal{X}} \|\xi\|_{W^{\frac{1}{\theta},\theta'}(\Gamma_1)},
\]
that is,
\[
\gamma_1u \in (W^{\frac{1}{\theta},\theta'}(\Gamma_1))'=W^{-\frac{1}{\theta},\theta}(\Gamma_1)
\]
and
\[
\|\gamma_1u\|_{W^{-\frac{1}{\theta},\theta}(\Gamma_1)} \leqslant c\|u\|_{\mathcal{X}}, \quad \forall u  \in \mathcal{D}(\overline{\Omega}).
\]
Now, the results follows by density.

\eproof	

\section{Existence and uniqueness of global solution}

This section concerns the proof of the Theorem \ref{theorem1}, Theorem \ref{theorem2} and Theorem \ref{theorem3}. 
\subsection{Proof of Theorem \ref{theorem1}}

We use Faedo-Galerkin method with compactness arguments and ideas used by MiIla Miranda, Lourêdo and Medeiros \cite{Carrier}. 
	\\
	{\bf Approximate Problem}.
	Let $(w_p)_{p\in\mathbb{N}}$ be a basis of the separable Banach space $V$. Let \linebreak$V_m=[w_1, \cdots, w_m]$ be the subspace generated by the $m$ first vectores  $w_1,w_2, \cdots,w_m.$ Consider 
	$$
	u_{m}(t)=\sum_{j=1}^mg_{jm}(t)w_j, \quad  \quad v_{m}(t)=\sum_{j=1}^mh_{jm}(t)w_j 
	$$
	such that $u_{m}$ and $v_{m}$ are approximate solutions of problem \eqref{int}, that is,
	\begin{equation}\label{EEEc1}
	\left|
	\begin{array}{l}
	(u''_{m}(t), w_j) + ((u_{m}(t), w_j)) + \displaystyle\int_{\Gamma_1} \delta u'_{m}(t)w_j d\Gamma+ \displaystyle\int_{\Omega}|u_{m}(t)|^{\rho}|v_{m}(t)|^{\rho}v_{m}(t) w_j dx =0 \vspace{0.3cm} \\
	(v''_{m}(t), w_l) + ((v_{m}(t), w_l)) +\displaystyle\int_{\Gamma_1} \delta v'_{m}(t)w_l d\Gamma+ \displaystyle\int_{\Omega}|u_{m}(t)|^{\rho}u_{m}(t)|v_{m}(t)|^{\rho} w_ldx =0, \vspace{0.3cm}\\
	u_m(0)=u_{0m} \to u^0 ~\mbox{in}~  V  \quad \text{and} \quad u'_m(0)=u_{1m} \to u^1 ~\mbox{in}~ L^{2}(\Omega)  \vspace{0.3cm} \\
	v_m(0)=v_{0m} \to v^0 ~\mbox{in}~ V \quad \text{and} \quad v'_m(0)=v_{1m} \to v^1 ~\mbox{in}~ L^{2}(\Omega)
	\end{array}
	\right.
	\end{equation}
	for all $j=1,2,\ldots,m$ and for all $l=1,2,\dots,m$. 
	
	The above finite-dimensional system has solutions $\{u_{m}(t), v_{m}(t)\}$ defined on $[0,t_{m})$. The following estimate allows us to extend this solution to the interval $[0,\infty)$.  
	\begin{remark}\label{obs1}
		
		We prove initially that the integral
		\begin{equation}\label{dado1}
		\int_{\Omega}|u_{m}(t)|^{\rho}|v_{m}(t)|^{\rho}v_{m}(t)w_j dx
		\end{equation}
		makes sense. Indeed, firstly we note that $w_j \in L^{q}(\Omega)$.
	If $3\leqslant n\leqslant6$ and we use the item $(ii)$ of Remark \ref{obsH1}. We obtain, noting that $q'=\frac{2n}{n+2}$
		\begin{equation*}\label{bolac}
		\begin{split}
		&\displaystyle\int_{\Omega} |u_{m}(t)|^{\rho q'}|v_{m}(t)|^{\rho q'}|v_{m}(t)|^{q'} dx = \displaystyle\int_{\Omega} |u_{m}(t)|^{\frac{2n\rho}{n+2}}|v_{m}(t)|^{\frac{2n\rho}{n+2}}|v_{m}(t)|^{\frac{2n}{n+2}} dx\\
		& \leqslant \left( \displaystyle\int_{\Omega} |u_{m}(t)|^{\frac{8n\rho}{n+2}} dx\right)^{\frac{1}{4}} \left( \displaystyle\int_{\Omega} |v_{m}(t)|^{\frac{8n\rho}{n+2}}dx \right)^{\frac{1}{4}}\left( \displaystyle\int_{\Omega} |v_{m}(t)|^{\frac{4n}{n+2}}dx \right)^{\frac{1}{2}}
		\\
		&= \|u_{m}(t)\|_{L^{\frac{8n\rho}{n+2}}(\Omega)}^{\frac{2n\rho}{n+2}} \|v_{m}(t)\|_{L^{\frac{8n\rho}{n+2}}(\Omega)}^{\frac{2n\rho}{n+2}}\|v_{m}(t)\|_{L^{\frac{4n}{n+2}}(\Omega)}^{\frac{2n}{n+2}}
		\\
		& \leqslant C \|u_{m}(t)\|^{\frac{2n\rho}{n+2}} \|v_{m}(t)\|^{\frac{2n\rho}{n+2}}\|v_{m}(t)\|^{\frac{2n}{n+2}},
		\end{split}
		\end{equation*}
		for some positive constant $C$. Therefore the above integral \eqref{dado1} makes sense. Similar considerations for the integral
		\[
		\int_{\Omega}|u_{m}(t)|^{\rho}u_{m}(t)|v_{m}(t)|^{\rho}w_l dx.
		\]
	\end{remark}
	{\bf A Priori Estimates.}
	Multiplying both of sides of $(\ref{EEEc1})_1$ by $g'_{jm}(t)$ and adding from $j=1$ to $j=m.$ We obtain
	\begin{equation}\label{EEEc3}
	\begin{split}
	&\dfrac{1}{2} \dfrac{d}{dt}|u'_{m}(t)|^2 + \dfrac{1}{2} \dfrac{d}{dt}\|u_{m}(t)\|^2 + \int_{\Gamma_1} \delta [u'_{m}(t)]^2d\Gamma \\
	&+\int_{\Omega} (|u_{m}(t)|^{\rho}u'_{m}(t))(|v_{m}(t)|^{\rho}v_{m}(t)) dx = 0.
	\end{split}
	\end{equation}
	We observe that
	\begin{equation}\label{EEEc4}
	\frac{d}{dt}(|u_{m}(t)|^{\rho}u_{m}(t))=(\rho +1)|u_{m}(t)|^{\rho}u'_{m}(t)
	\end{equation}
	Taking into account \eqref{EEEc4} in \eqref{EEEc3}, we obtain
	
	\begin{equation}\label{EEEc5}
	\begin{split}
	&\dfrac{1}{2} \dfrac{d}{dt}|u'_{m}(t)|^2 + \dfrac{1}{2} \dfrac{d}{dt}\|u_{m}(t)\|^2 + \int_{\Gamma_1} \delta [u'_{m}(t)]^2d\Gamma 
	\\
	& + \frac{1}{\rho +1} \int_{\Omega} \frac{d}{dt}(|u_{m}(t)|^{\rho}u_{m}(t))(|v_{m}(t)|^{\rho}v_{m}(t)) dx = 0.
	\end{split}
	\end{equation}
	Similarly multiplying both of sides of $\eqref{EEEc1}_2$ by $h'_{jm}(t)$, we obtain
	\begin{equation}\label{EEEc6}
	\begin{split}
	&\dfrac{1}{2} \dfrac{d}{dt}|v'_{m}(t)|^2 + \dfrac{1}{2} \dfrac{d}{dt}\|v_{m}(t)\|^2+ \int_{\Gamma_1} \delta [v'_{m}(t)]^2d\Gamma 
	\\
	& + \frac{1}{\rho +1} \int_{\Omega} (|u_{m}(t)|^{\rho}u_{m}(t))\frac{d}{dt}(|v_{m}(t)|^{\rho}v_{m}(t)) dx = 0.
	\end{split}
	\end{equation}
	Adding $\eqref{EEEc5}$ and $\eqref{EEEc6}$, we have
	\[
	\begin{split}
	&\dfrac{1}{2} \dfrac{d}{dt}\Big\{|u'_{m}(t)|^2 + \|u_{m}(t)\|^2 + |v'_{m}(t)|^2 + \|v_{m}(t)\|^2 \Big\} + \int_{\Gamma_1} \delta [u'_{m}(t)]^2d\Gamma \\
	&+ \int_{\Gamma_1} \delta [v'_{m}(t)]^2d\Gamma + \dfrac{1}{\rho +1} \dfrac{d}{dt} \displaystyle \int_{\Omega} (|u_{m}(t)|^{\rho}u_{m}(t))(|v_{m}(t)|^{\rho}v_{m}(t)) dx = 0.
	\end{split}
	\]
	Integrating the above expression from $0$ to $t$ with $t < t_{m}$, and using the hypothesis on $\delta$, we obtain
	
	\begin{equation}\label{EEEc8}
	\begin{split}
	&\dfrac{1}{2}\Big\{|u'_{m}(t)|^2 + \|u_{m}(t)\|^2 + |v'_{m}(t)|^2 + \|v_{m}(t)\|^2 \Big\} +\delta_0\int_{0}^{t}\int_{\Gamma_1} [u'_{m}(s)]^2d\Gamma   ds\\ 
	&+\delta_0 \int_{0}^{t}\int_{\Gamma_1} [v'_{m}(s)]^2d\Gamma ds+ \frac{1}{\rho +1}  \int_{\Omega} (|u_{m}(t)|^{\rho}u_{m}(t))(|v_{m}(t)|^{\rho}v_{m}(t)) dx  \\
	&\leqslant \dfrac{1}{2}\Big\{|u_{1m}|^2 + \|u_{0m}\|^2 + |v_{1m}|^2 + \|v_{0m}\|^2 \Big\}   
	+ \frac{1}{\rho +1}  \int_{\Omega} (|u_{0m}(x)|^{\rho}u_{0m}(x))(|v_{0m}(x)|^{\rho}v_{0m}(x)) dx.
	\end{split}
	\end{equation}
	From Young's inequality, we get
	\[
	\begin{split}
	&\left| \frac{1}{\rho +1}  \int_{\Omega}  (|u_{m}(t)|^{\rho}u_{m}(t))(|v_{m}(t)|^{\rho}v_{m}(t)) dx \right | \\
       &\leqslant \frac{1}{2(\rho +1)} \left\{ \|u_{m}(t)\|_{L^{2(\rho +1)}(\Omega)}^{2(\rho +1)} dx +  \|v_{m}(t)\|_{L^{2(\rho +1)}(\Omega)}^{2(\rho +1)} dx \right \}.
	\end{split}
	\]
	Now using the fact $V \hookrightarrow L^{2(\rho +1)}(\Omega)$, (see \eqref{H1}), we derive
	\begin{equation}\label{Zc1}
	\begin{split}
	\left| \frac{1}{\rho +1}  \int_{\Omega}  (|u_{m}(t)|^{\rho}u_{m}(t))(|v_{m}(t)|^{\rho}v_{m}(t)) dx \right| 
	\leqslant N \Big[ \|u_{m}(t)\|^{2(\rho +1)}  + \|v_{m}(t)\|^{2(\rho +1)}\Big],
	\end{split}
	\end{equation}
	where $N$ was defined in \eqref{Nc}. Analogously we obtain 
	\begin{equation}\label{Zc2}
	\left| \frac{1}{\rho +1}  \int_{\Omega}  (|u_{0m}|^{\rho}u_{0m})(|v_{0m}|^{\rho}v_{0m}) dx \right| 
	\leqslant N \Big[ \|u_{0m}\|^{2(\rho +1)}  + \|v_{0m}\|^{2(\rho +1)}\Big].
	\end{equation}
	Substituting \eqref{Zc1} and \eqref{Zc2} in \eqref{EEEc8} we obtain
	
	\begin{equation}\label{Zc3}
	\begin{split}
	&\dfrac{1}{2}\Big\{|u'_{m}(t)|^2 + \|u_{m}(t)\|^2 + |v'_{m}(t)|^2 + \|v_{m}(t)\|^2 \Big\}- N \Big\{ \|u_{m}(t)\|^{2(\rho +1)} + \|v_{m}(t)\|^{2(\rho +1)} \Big\} 
	\\
	&+\delta_0\int_{0}^{t}\int_{\Gamma_1}[u'_{m}(s)]^2d\Gamma ds+\delta_0\int_{0}^{t}\int_{\Gamma_1}[v'_{m}(s)]^2d\Gamma  ds\leqslant \dfrac{1}{2}\Big\{|u_{1m}|^2 + \|u_{0m}\|^2+|v_{1m}|^2 +\|v_{0m}\|^2 \Big\}  
	\\ 
	&+ N \Big\{\|u_{0m}\|^{2(\rho +1)} + \|v_{0m}\|^{2(\rho +1)} \Big\}.  
	\end{split}
	\end{equation}
	By hypotheses and convergences $\eqref{EEEc1}$, for small $\eta>0$, we obtain
	\begin{equation}\label{ideal}
	\begin{split}
	&\|u_{0m}\|<\|u^0\| + \eta < \lambda^{*}, \quad  \|v_{0m}\|<\|v^0\| + \eta < \lambda^{*}, ~~\forall m\geqslant m_0 \\ \\
	L_m&=\frac{1}{2}[|u_{1m}|^2+|v_{1m}|^2]+\frac{1}{2}[\|u_{0m}\|^2+\|v_{0m}\|^2] \\ &+N[\|u_{0m}\|^{2(\rho+1)}+\|v_{0m}\|^{2(\rho+1)}] < L + \eta< \frac{1}{2}(\lambda^{*})^2,~~\forall m\geqslant m_0 
	\end{split}
	\end{equation}
	where $L$ was introduced in \eqref{Hc22}. Therefore from \eqref{Zc3} and \eqref{ideal}, we have for small $\eta >0$,
	\begin{equation}\label{Zc4}
	\begin{split}
	&\dfrac{1}{2}\Big\{|u'_{m}(t)|^2 + \|u_{m}(t)\|^2 + |v'_{m}(t)|^2 + \|v_{m}(t)\|^2 \Big\}- N \Big\{ \|u_{m}(t)\|^{2(\rho +1)} + \|v_{m}(t)\|^{2(\rho +1)} \Big\} \\
	&+\delta_0\int_{0}^{t}\int_{\Gamma_1}[u'_{m}(s)]^2d\Gamma ds+\delta_0\int_{0}^{t}\int_{\Gamma_1}[v'_{m}(s)]^2d\Gamma ds
	<  L + \eta < \frac{1}{4}(\lambda^{*})^2, ~ \forall m \geqslant m_0,
	\end{split}
	\end{equation}
	Motivated by \eqref{Zc4}, we set the function
	\begin{equation*}\label{functioncJ}
	J(\lambda)=\dfrac{1}{4} \lambda^2 -  N\lambda^{2(\rho +1)}, \quad \lambda \geqslant 0.
	\end{equation*}
	Then \eqref{Zc4} provides
	
	\begin{equation}\label{Zc5}
	\begin{split}
	&\dfrac{1}{2}|u'_{m}(t)|^2 +\dfrac{1}{4} \|u_{m}(t)\|^2 + J(\|u_{m}(t)\|) + \dfrac{1}{2} |v'_{m}(t)|^2 + \dfrac{1}{4}\|v_{m}(t)\|^2  + J(\|v_{m}(t)\|)
	\\
	& +\delta_0\int_{0}^{t}\int_{\Gamma_1}[u'_{m}(s)]^2d\Gamma ds+\delta_0\int_{0}^{t}\int_{\Gamma_1}[v'_{m}(s)]^2d\Gamma ds <  L + \eta < \frac{1}{4}(\lambda^{*})^2, \quad \forall m \geqslant m_0.
	\end{split}
	\end{equation}
	
	In order to obtain an a priori estimates for the approximate solutions $u_m$ and $v_m$, we need that the left side of \eqref{Zc5} 
	would be non-negative. It is possible if $J(\|u_{m}(t)\|)$ and $J(\|v_{m}(t)\|)$ are non-negative. In the next result we prove that if the hypothesis \eqref{restriction} is satisfied then    
	\[
	J(\|u_{m}(t)\|)\geqslant 0, \ J(\|v_{m}(t)\|)\geqslant 0, \quad \forall \ t \in [0,\infty).
	\]
	\begin{remark}\label{obs2}
		Note that
		\[
		J(\lambda)=\dfrac{1}{4} \lambda^2 -  N\lambda^{2(\rho +1)}\geqslant0, \quad \forall ~0\leqslant \lambda \leqslant \lambda^{*}.
		\]	
		This fact is consequence of 
		\[
		J(\lambda)=\lambda^2 \left(\dfrac{1}{4} -  N\lambda^{2\rho}\right), \quad \lambda \geqslant0.
		\]	
	\end{remark}	
	\begin{lemma}\label{Lc1}
		Consider $u^0,  v^0 \in V$ and $u^1,  v^1 \in L^2(\Omega)$ such that 
		$$
		\|u^0\|,\ \|v^0\| < \lambda^{*} 
		$$
		and
		$$
		L_1 < \dfrac{1}{4}(\lambda^{*})^2
		$$
		where $\lambda^{*}$ and $L$ were defined, respectively, in \eqref{Hc21} and \eqref{Hc22}.
		Then
		\begin{equation*}\label{oc}
		\|u_{m}(t)\|< \lambda^{*} \quad \text{and} \quad \|v_{m}(t)\| < \lambda^{*}, \quad \forall t \in [0, t_{m}),\ \forall m \geqslant m_0.
		\end{equation*}
	\end{lemma}
	\proof
		We fix $m \geqslant m_0$. We show the lemma by contradiction argument. Thus assume that there exists 
		$t_1 \in (0, t_{m})$ or $t_2 \in (0, t_{m})$ such that 
		\[
		\|u_{m}(t_1)\| \geqslant \lambda^{*} \quad \mbox{or} \quad \|v_{m}(t_2)\| \geqslant \lambda^{*}.
		\]
		There are two possibilities, which are 
		\begin{equation}\label{possibility}
		\begin{split}
		& 1) \quad \|u_{m}(t_1)\| \geqslant \lambda^{*} \quad \mbox{and} \quad \|v_{m}(t_2)\| \geqslant \lambda^{*}, 
		\\
		& 2)\quad  \|u_{m}(t_1)\| \geqslant \lambda^{*} \quad \mbox{and}\quad \|v_{m}(t)\| <\lambda^{*} \ \forall \ t \in [0,\infty).
		\end{split}
		\end{equation}
		Assume that occurs possibility $1)$. Then by intermediate Value Theorem there exists $\tau_1  \in (0, t_{m})$ and $\tau_2  \in (0, t_{m})$ such that
		\[
		\|u_{m}(\tau_1)\| = \lambda^{*} \quad \mbox{and} \quad \|v_{m}(\tau_2)\| = \lambda^{*}.
		\]
		Set
		\[
		t_1^{*}=\inf\{\tau\in (0, t_{m}) ; \|u_{m}(\tau)\| = \lambda^{*}\}
		\]	
		\[
		t_2^{*}=\inf\{\tau\in (0, t_{m}) ; \|v_{m}(\tau)\| = \lambda^{*}\}.
		\]	
		By continuity of $\|u_{m}(t)\|$ and $\|v_{m}(t)\|$, we obtain 
		\[
		\|u_{m}(t_1^{*})\| = \lambda^{*} \quad \mbox{and} \quad \|v_{m}(t_2^{*})\| = \lambda^{*}.
		\]
		From $\eqref{ideal}_1$ it follows that $t_1^{*}>0$ and $t_2^{*}>0$. Thus 
		\[
		\|u_{m}(t)\| < \lambda^{*} \quad \mbox{for} \quad 0 \leqslant t <t_1^{*}
		\]
		\[
		\|v_{m}(t)\| < \lambda^{*} \quad \mbox{for} \quad 0 \leqslant t <t_2^{*}.
		\]
		Therefore by Remark \ref{obs2}, we get
		\[
		J(\|u_{m}(t)\|)\geqslant0  \quad \mbox{for} \quad 0 \leqslant t <t_1^{*}
		\]
		\[
		J(\|v_{m}(t)\|) \geqslant 0 \quad \mbox{for} \quad 0 \leqslant t <t_2^{*}.
		\]
		Assume $t_1^{*} \leqslant t_2^{*}$. Similar arguments if $t_2^{*} \leqslant t_1^{*}$. Return to expression \eqref{Zc5}. Then
		\[
		\dfrac{1}{4} \|u_{m}(t)\|^2 + J(\|u_{m}(t)\|) +  \dfrac{1}{4}\|v_{m}(t)\|^2  + J(\|v_{m}(t)\|) \leqslant L + \eta < \frac{1}{4}(\lambda^{*})^2  \quad 0 \leqslant t <t_1^{*}
		\]
		So
		\[
		\dfrac{1}{4} \|u_{m}(t)\|^2 \leqslant L + \eta < \frac{1}{4}(\lambda^{*})^2  \quad  0 \leqslant t <t_1^{*},~~\forall m\geqslant m_0 .
		\]
		Taking the limit as $t \to t_1^{*}, ~ 0 < t <t_1^{*}$ in this inequality we obtain a contradiction. This prove the part $1$ of \eqref{possibility}.
		
		The proof of possibility $2)$ of \eqref{possibility} follows by applying the arguments used in part $1)$ to $\|u_{m}(t_1)\| \geqslant \lambda^{*}$ and this conclude the proof of the lemma.
	\eproof
	
	By by Lemma \eqref{Lc1} we have
	$$
	\|u_{m}(t)\| < \lambda^{*}, \quad \|v_{m}(t)\| < \lambda^{*} \quad \forall ~ 0 \leqslant t < \infty, \quad \forall ~ m \geqslant m_0. 
	$$
	Consequently
	$$
	J(\|u_{m}(t)\|) \geqslant 0,~J(\|v_{m}(t)\|)\geqslant0, \quad \forall ~t \in [0,\infty).
	$$
	Therefore
	\begin{equation}\label{Dc5}
	\begin{split}
	&\dfrac{1}{2}|u'_{m}(t)|^2 +\dfrac{1}{4} \|u_{m}(t)\|^2+ \dfrac{1}{2} |v'_{m}(t)|^2 + \dfrac{1}{4}\|v_{m}(t)\|^2  
	\\
	&+\delta_0\int_{0}^{t}\int_{\Gamma_1}[u'_{m}(s)]^2d\Gamma ds+\delta_0\int_{0}^{t}\int_{\Gamma_1}[v'_{m}(s)]^2d\Gamma ds
	\leqslant L + \eta < \dfrac{1}{4}(\lambda^{*})^2,
	\end{split}
	\end{equation}
	for all $t \in [0, \infty)$ and for all $m \geqslant m_0$. By $\eqref{Dc5}$ we obtain
	\begin{equation}\label{EEEc10}
	\left|
	\begin{array}{l}
	(u_{m}), (v_{m})~\mbox{are bounded in } ~L^{\infty}(0, \infty; V), ~\forall m\geqslant m_0 ;\vspace{0.2cm}\\
	(u'_{m}), (v'_{m})~\mbox{are bounded in } ~L^{\infty}(0, \infty; L^2(\Omega)),~\forall m\geqslant m_0; \vspace{0.2cm}\\
	(u'_{m}), (v'_{m})~\mbox{are bounded in} ~L^{2}(0, \infty; L^2(\Gamma_1)), ~\forall m\geqslant m_0.
	\end{array}
	\right.
	\end{equation}
	With similar arguments used in the Remark \ref{obs1} we obtain
	$$
	\| |u_{m}(t)|^{\rho}  |v_{m}(t)|^{\rho}v_{m}(t)|\|_{L^{q'}(\Omega)}  \leqslant C, \quad \forall ~ m \geqslant m_0.
	$$
	where the constant $C>0$ is independent of $t$ and $m.$ It follows that
	\begin{equation}\label{Cc2}
	(|u_{m}|^{\rho}|v_{m}|^{\rho}v_{m})~\mbox{is bounded in } ~L^{\infty}(0, \infty; L^{q'}(\Omega))), \quad \forall ~ m \geqslant m_0.
	\end{equation}
	In similar way, we find
	\begin{equation}\label{Cc21}
	(|u_{m}|^{\rho}u_{m}|v_{m}|^{\rho})~\mbox{is bounded in } ~L^{\infty}(0, \infty; L^{q'}(\Omega))), \quad \forall ~ m \geqslant m_0.
	\end{equation}
	{\bf Passage to the limit in} $m$.
	Estimates \eqref{EEEc10}, \eqref{Cc2} and \eqref{Cc21} allow us, by induction and diagonal process, to obtain a subsequences of $(u_{m})$ and $(v_{m})$, still denoted by $(u_{m})$ and $(v_{m})$, and functions $u,v: \Omega \times [0,\infty) \to \mathbb{R}$, such that
	
	\begin{equation}\label{Zc7}
	\left|
	\begin{array}{l}
	u_{m} \to u \quad \mbox{and} \quad  v_{m} \to v \quad \mbox{weak star in} \quad  L^{\infty}(0,\infty, V), \vspace{0.3cm}\\
	u'_{m} \to u' \quad \mbox{and} \quad  v'_{m} \to v' \quad \mbox{weak star in} \quad  L^{\infty}(0,\infty, L^2(\Omega)), \vspace{0.3cm}\\
	u'_{m} \to u' \quad \mbox{and} \quad  v'_{m} \to v' \quad \mbox{weak in} \quad  L^{2}(0,\infty, L^2(\Gamma_1)), \vspace{0.3cm}\\
	|u_{m}|^{\rho}|v_{m}|^{\rho}v_{m} \to \xi \quad \mbox{weak star in} \quad  L^{\infty}(0,\infty, L^{q'}(\Omega)),  \vspace{0.3cm}\\
	|u_{m}|^{\rho}u_{m}|v_{m}|^{\rho} \to \zeta \quad \mbox{weak star in} \quad  L^{\infty}(0,\infty, L^{q'}(\Omega)). 
	\end{array}
	\right.
	\end{equation}
	
	We must show that $\xi = |u|^{\rho}|v|^{\rho}v$ and $\zeta = |u|^{\rho}u|v|^{\rho}$. 
	
	Consider $T>0$ fixed but arbitrary. By convergences $\eqref{Zc7}_1$ and $\eqref{Zc7}_2$ and noting that
	$V \stackrel{compact}\hookrightarrow L^2(\Omega) $, we obtain by Aubin-Lions Theorem that there are subsequences of $(u_{m}), ~(v_{m})$, which we still denoted by $(u_{m}), ~(v_{m})$, respectively, such that 
	\begin{equation}\label{Zc8}
	\left|
	\begin{array}{l}
	u_{m} \to u \quad \mbox{strong in} \quad   L^{\infty}(0,T, L^2(\Omega)),  \vspace{0.3cm} 
	\\
	v_{m} \to v \quad \mbox{strong in} \quad  L^{\infty}(0,T, L^2(\Omega)).
	\end{array}
	\right.
	\end{equation}
	By \eqref{Zc8} there are subsequences of $(u_{m})$ and $(v_{m})$ such that 
	\begin{equation}\label{Zc9}
	\left|
	\begin{array}{l}
	u_{m} \to u \quad \mbox{a.e. in} \quad  \Omega \times (0,T), \vspace{0.3cm} 
	\\
	v_{m} \to v \quad \mbox{a.e. in} \quad  \Omega \times (0,T).
	\end{array}
	\right.
	\end{equation}
	By \eqref{Zc9} we have that
	\begin{equation*}\label{Zc10}
	\left|
	\begin{array}{l}
	|u_{m}|^{\rho} \to |u|^{\rho} \quad \mbox{a.e. in} \quad   \Omega \times (0,T), \vspace{0.3cm} 
	\\
	|v_{m}|^{\rho}v_{m}  \to |v|^{\rho}v \quad \mbox{a.e. in} \quad  \Omega \times (0,T).
	\end{array}
	\right.
	\end{equation*}
	Therefore
	\begin{equation}\label{Zc11}
	|u_{m}|^{\rho}|v_{m}|^{\rho}v_{m}  \to |u|^{\rho}|v|^{\rho}v  \quad \mbox{a.e. in} \quad \Omega \times (0,T).
	\end{equation}
	From $\eqref{Cc2},$ $\eqref{Zc11}$ and of Lions' Lemma we obtain
	\begin{equation*}
	|u_{m}|^{\rho}|v_{m}|^{\rho}v_{m}  \to |u|^{\rho}|v|^{\rho}v  \quad \mbox{weak in} \quad  L^{2}(0,T,L^{q'}(\Omega)) .
	\end{equation*}
	In similar way, we find
	\begin{equation*}
	|u_{m}|^{\rho}u_{m}|v_{m}|^{\rho}  \to |u|^{\rho}u|v|^{\rho}\quad \mbox{weak in}  \quad   L^{2}(0,T,L^{q'}(\Omega)) .
	\end{equation*}
	By a diagonal process we obtain
	\begin{equation}\label{Zc12}
	|u_{m}|^{\rho}|v_{m}|^{\rho}v_{m}  \to |u|^{\rho}|v|^{\rho}v  \quad \mbox{weak star in} \quad  L_{loc}^{\infty}(0,\infty,L^{q'}(\Omega)) .
	\end{equation}
	In similar way, we find
	\begin{equation}\label{Zc121}
	|u_{m}|^{\rho}u_{m}|v_{m}|^{\rho}  \to |u|^{\rho}u|v|^{\rho}\quad \mbox{weak star in}  \quad   L_{loc}^{\infty}(0,\infty,L^{q'}(\Omega)) .
	\end{equation}
	By \eqref{Zc7} and \eqref{Zc12}, \eqref{Zc121} we have $\xi=|u|^{\rho}|v|^{\rho}v$ and $\zeta=|u|^{\rho}u|v|^{\rho}$.
	
	Multiplying both sides of the approximate equation $\eqref{EEEc1}_1$ by $\varphi \in \mathcal{D}(0,\infty)$,
	integrating in $[0,\infty)$, using the convergences \eqref{Zc7} and noting that $V_m$ is dense in $V$ we obtain 
	\begin{equation}\label{aca}
	\begin{split}
	&\int_{0}^{\infty}(u''(t), w) \varphi(t) dt  +\int_{0}^{\infty} ((u(t), w)) \varphi(t) dt +\int_{0}^{\infty} \int_{\Gamma_1}\delta u'(t) w \varphi(t) d\Gamma dt 
	\\ 
	&+ \int_{0}^{\infty} ( |u(t)|^{\rho}|v(t)|^{\rho}v(t),w ) \varphi(t) dt = 0,~ \forall  w \in V, ~\forall \varphi \in \mathcal{D}(0,\infty).
	\end{split}
	\end{equation}
	Since $V$ is dense in $L^2(\Omega)$ it follows that \eqref{aca} is true for all $v \in L^2(\Omega)$.
	
	In similar way, we find
	\begin{equation*}
	\begin{split}
	&\displaystyle\int_{0}^{\infty}(v''(t), z)\varphi(t) dt  +\int_{0}^{\infty} ((v(t), z)) \varphi(t) dt + \int_{0}^{\infty} \int_{\Gamma_1}\delta u'(t) z \varphi(t) d\Gamma dt
	\\ 
	&+ \displaystyle\int_{0}^{\infty} ( |u(t)|^{\rho}u(t)|v(t)|^{\rho},z )\varphi(t)dt = 0, ~ \forall  z \in V,~\forall \varphi \in \mathcal{D}(0,\infty).
	\end{split}
	\end{equation*}
	Taking in \eqref{aca} $w \in \mathcal{D}(\Omega) \subset V$, it follows that
	\begin{equation}\label{Gc}
	u''-\Delta u + |u|^{\rho}|v|^{\rho}v = 0 \quad \mbox{in} \quad \mathcal{D}'(\Omega \times(0,\infty)).
	\end{equation}
	In similar way
	\[
	v''-\Delta v + |u|^{\rho}u|v|^{\rho} = 0 \quad \mbox{in} \quad \mathcal{D}'(\Omega \times(0,\infty)).
	\]
Let $T>0$ fix. Note that $u' \in L^{2}(0,T;L^2(\Omega))\hookrightarrow L^{2}(0,T;L^{q'}(\Omega))$ then $u'' \in H^{-1}(0,T;L^{q'}(\Omega))$. Since $|u|^{\rho}|v|^{\rho}v \in L^{2}(0,T;L^{q'}(\Omega))\hookrightarrow H^{-1}(0,T;L^{q'}(\Omega))$ then by \eqref{Gc} we have 
$$
-\Delta u \in H^{-1}(0,T;L^{q'}(\Omega)).
$$ 
Therefore
	\begin{equation*}\label{e1}
	u''-\Delta u + |u|^{\rho}|v|^{\rho}v = 0 \quad \mbox{in} \quad H^{-1}(0,T;L^{q'}(\Omega)).
	\end{equation*}
	In similar way
	\[
	v''-\Delta v + |u|^{\rho}u|v|^{\rho} = 0 \quad \mbox{in} \quad H^{-1}(0,T;L^{q'}(\Omega)).
	\]
	As $u \in L^{2}(0,T;V)$ and $\Delta u \in H^{-1}(0,T;L^{q'}(\Omega))$ then by Theorem \ref{tracoX} with $\theta=q'$ we obtain 
	\begin{equation}\label{tracoc}
	\frac{\partial u}{\partial \nu} \in H^{-1}((0,T;W^{-\frac{1}{q'},q'}(\Gamma_1))
	\end{equation}
	Multiplyng both sides of \eqref{Gc} by $w\varphi$ with $w \in L^2(\Omega)$ and $\varphi \in \mathcal{D}(0,\infty)$, integrating on $\Omega \times (0,\infty)$ and using \eqref{tracoc} and Green's formula
	\begin{equation*}
	\begin{split}
	&\displaystyle\int_{0}^{\infty}(u''(t), w) \varphi(t) dt  +\int_{0}^{\infty} ((u(t), w)) \varphi(t) dt - \int_{0}^{\infty} \left\langle \frac{\partial u(t)}{\partial \nu},w \right\rangle \varphi(t)d\Gamma dt
	\\ 
	&+ \displaystyle\int_{0}^{\infty} ( |u(t)|^{\rho}u(t)|v(t)|^{\rho},w )\varphi(t) dt = 0, ~ \forall  w \in V ~\forall \varphi \in \mathcal{D}(0,\infty),
	\end{split}
	\end{equation*}
	where $\langle \cdot,\cdot\rangle$ denotes the duality paring between $W^{-\frac{1}{q'},q'}(\Gamma_1)$ and $W^{\frac{1}{q'},q}(\Gamma_1)$. Comparing this lest equation with \eqref{aca}, we obtain
	\[
	\int_{0}^{\infty}\left\langle \frac{\partial u(t)}{\partial \nu}+\delta u'(t),w \right\rangle \varphi(t) d\Gamma dt=0, ~ \forall  w \in L^2(\Omega) ~\forall \varphi \in \mathcal{D}(0,\infty)
	\]
	Therefore
	\[
	\frac{\partial u}{\partial \nu}+\delta u'=0 \quad \mbox{in} \quad W^{-\frac{1}{q'},q'}(\Gamma_1).
	\]
	In similar way
	\[
	\frac{\partial v}{\partial \nu}+\delta v'=0 \quad \mbox{in} \quad W^{-\frac{1}{q'},q'}(\Gamma_1).
	\]
	Snce $\delta u' \in L^{2}(0,\infty, L^2(\Gamma_1))$, then
	\begin{equation*}\label{f1c}
	\frac{\partial u}{\partial \nu}+\delta u'=0 \quad \mbox{in} \quad L^{2}(0,T;L^{2}(\Gamma_1)).
	\end{equation*}
	In similar way
	\begin{equation*}\label{f2c}
	\frac{\partial v}{\partial \nu}+\delta v'=0 \quad \mbox{in} \quad L^{2}(0,T;L^{2}(\Gamma_1)).
	\end{equation*}
	{\bf Initial Conditions.} Using a standard argument, we can verify the initial conditions.

\subsection{Proof of Theorem \ref{theorem2}}

As the saparable space $W_{\Gamma_0}^{1,\theta'}(\Omega)$ is dense in $V$ and dense in $L^{\theta'}(\Omega)$ and $W_{\Gamma_0}^{1,\theta'}(\Omega)\hookrightarrow V\cap L^{\theta'}(\Omega)$ by Proposition \ref{p1} and Proposition \ref{p2} we have that $V\cap L^{\theta'}(\Omega)$ is a separable Banach space. Thus, taking a basis $(w_l)_{l \in \mathbb{N}}$ in $V\cap L^{\theta'}(\Omega)$ where $\dfrac{1}{\theta}+\dfrac{1}{\theta'}=1$ and using similar arguments to those used in the Theorem \ref{theorem1} we show the Theorem \ref{theorem2}.

\subsection{Proof of Theorem \ref{theorem3}}
The proof of Theorem \ref{theorem3} will be done by applying the Faedo-Galerkin method with a special basis of $V\cap H^2(\Omega)$. But to this we need of the following proposition.
\begin{proposition}\label{propos}
	Let us consider $f\in L^2(\Omega)$ and $g \in H^{\frac{1}{2}}(\Gamma_1)$. Then, the solution $u$ of the boundary value problem:
	\begin{equation*}
	\left|
	\begin{array}{l}
	-\Delta w=f \quad \mbox{in} \quad \Omega \\
	w=0 \quad \mbox{on} \quad \Gamma_0 \\
	\dfrac{\partial w}{\partial \nu}=g \quad \mbox{on} \quad \Gamma_1
	\end{array}
	\right.
	\end{equation*}
	belongs to $V \cap H^2(\Omega)$ and 
	\[
	\|w\|_{H^2(\Omega)} \leqslant [|f|+\|g\|_{H^{\frac{1}{2}}(\Gamma_1)}].
	\]
\end{proposition}	
\begin{proof}
	For the proof see Milla Miranda and Medeiros \cite[Proposition 1, p 49]{MMLA1996}. 
\end{proof}
\begin{proposition}\label{propos1}
	Suppose $u^0 \in V\cap H^{2}(\Omega), ~ u^1 \in V$ and 
	\[
	\frac{\partial u^0}{\partial \nu}+\delta(x) u^1=0 \quad \mbox{on} \quad \Gamma_1.
	\]
	Then, for each $\varepsilon$, there exists $w$ and $z$ in $V\cap H^{2}(\Omega)$ such that:
	\[
	\|w-u^0\|_{V\cap H^{2}(\Omega)} < \varepsilon,\quad \|z-u^1\| < \varepsilon \quad \mbox{and} \quad \frac{\partial w}{\partial \nu}+\delta(x) z=0 \quad \mbox{on} \quad \Gamma_1.
	\]
\end{proposition}
\begin{proof}
	For the proof see Milla Miranda and Medeiros \cite[Proposition 3, p. 50]{MMLA1996}.
\end{proof}
{\bf Approximate Problem}. From Proposition \ref{propos1}, we obtain sequences $(u_k^0), (v_k^0)$ and $(u_k^1), (v_k^1)$ of vectores of $V\cap H^2(\Omega)$ such that 

\begin{equation}\label{c1}
\begin{split}
&u_k^0 \to u^0 \quad \mbox{in} \quad \mbox{and} \quad v_k^0 \to v^0 \quad \mbox{in} \quad V \cap H^2(\Omega) \\
&u_k^1 \to u^1 \quad \mbox{in} \quad \mbox{and} \quad v_k^1 \to v^1 \quad \mbox{in} \quad V,
\end{split}
\end{equation}
and
\begin{equation*}\label{c2}
\frac{\partial u_k^0}{\partial \nu}+\delta u_k^1=0 \quad \mbox{and} \quad\frac{\partial v_k^0}{\partial \nu}+\delta v_k^1=0 \quad \mbox{on} \quad \Gamma_1 \quad \mbox{for all} \quad k \in \mathbb{N}.
\end{equation*}

Now we fix $k \in \mathbb{N}$ and consider the basis $\{w_1^k,w_2^k,w_3^k,w_4^k, \ldots\}$ of $V\cap H^2(\Omega)$ such that $u_k^0, v_k^0,u_k^1$ and $v_k^1$ belong to the subspace $[w_1^k,w_2^k,w_3^k,w_4^k]$ spanned by $w_1^k,w_2^k,w_3^k$ and $w_4^k$. For each $m \in \mathbb{N}$ we built the subspace $V_{m}^k=[w_1^k,w_2^k, \ldots, w_m^k]$. Consider 
\[
u_{km}(t)=\sum_{j=1}^mg_{jkm}(t)w_j^k, \quad  \quad v_{km}(t)=\sum_{j=1}^mh_{jkm}(t)w_j^k 
\]
such that $u_{km}$ and $v_{km}$ are approximate solutions of problem \eqref{int}, that is,
\begin{equation}\label{EEE1}
\left|
\begin{array}{l}
(u''_{km}(t), w) + ((u_{km}(t), w)) + \displaystyle\int_{\Gamma_1} \delta u'_{km}(t)w d\Gamma+ \displaystyle\int_{\Omega}|u_{km}(t)||v_{km}(t)|v_{km}(t) w dx =0 \vspace{0.3cm} \\
(v''_{km}(t), z) + ((v_{km}(t), z)) +\displaystyle\int_{\Gamma_1} \delta v'_{km}(t)z d\Gamma+ \displaystyle\int_{\Omega}|u_{km}(t)|u_{km}(t)|v_{km}(t)| z dx =0, \vspace{0.3cm}\\
u_{km}(0)=u_k^{0}, \quad u'_{km}(0)=u_k^{1} \vspace{0.3cm} \\
v_{km}(0)=v_k^{0}, \quad v'_{km}(0)=v_k^{1} 
\end{array}
\right.
\end{equation}
for all $w,z \in V_m^k$.  

The above finite-dimensional system has solutions $\{u_{km}(t), v_{km}(t)\}$ defined on $[0,t_{km})$. The following estimate allows us to extend this solution to the interval $[0,\infty)$.  
\begin{remark}
	Using similar arguments used in Theorem \ref{theorem1} with $\rho=1$ we prove that the integrals
	\[
	\int_{\Omega}|u_{km}(t)||v_{km}(t)|v_{km}(t)w dx \quad \mbox{and} \quad \int_{\Omega}|u_{km}(t)|u_{km}(t)|v_{km}(t)|z dx.
	\]
	makes sense. 
\end{remark}

{\bf First estimate.} To obtain the first estimate we apply similar arguments used in Theorem \ref{theorem1} with $\rho=1$. In this case, we replace the function $J$ by
\begin{equation*}\label{J1}
J_1(\lambda)= \frac{1}{4}\lambda^2-N_1\lambda^4,
\end{equation*}
where $N_1$ was defined in \eqref{N}. We also obtain the following lemma
\begin{lemma}\label{L1}
	Consider $u_0,  v_0 \in V\cap H^2(\Omega)$ and $u_1,  v_1 \in V$ such that 
	$$
	\|u_0\|,\ \|v_0\| < \lambda_1^{*} 
	$$
	and
	$$
	L_1 < \dfrac{1}{4}(\lambda_1^{*})^2
	$$
	where $\lambda_1^{*}$ and $L_1$ were defined, respectively, in \eqref{H21} and \eqref{H22}.
	Then
	\begin{equation*}\label{o}
	\|u_{km}(t)\|< \lambda_1^{*} \quad \text{and} \quad \|u_{km}(t)\| < \lambda_1^{*}, \quad \forall t \in [0, t_{km}),\ \forall k \geqslant k_0 ~\forall m.
	\end{equation*}
\end{lemma}
Therefore, we get 
\begin{equation}\label{EEE10}
\left|
\begin{array}{l}
(u_{km}), (v_{km})~\mbox{are bounded in } ~L^{\infty}(0, \infty; V); \quad \forall k\geqslant k_0, ~ \forall m\in \mathbb{N}\vspace{0.2cm}\\
(u'_{km}), (v'_{km})~\mbox{are bounded in } ~L^{\infty}(0, \infty; L^2(\Omega)),\quad \forall k\geqslant k_0, ~ \forall m\in \mathbb{N}\vspace{0.2cm}\\
(u'_{km}), (v'_{km})~\mbox{are bounded in} ~L^{2}(0, \infty; L^2(\Gamma_1)), \quad \forall k\geqslant k_0, ~ \forall m\in \mathbb{N}
\vspace{0.2cm}\\
(|u_{km}||v_{km}|v_{km})~\mbox{is bounded in } ~L^{\infty}(0, \infty; L^{2}(\Omega))), \quad \forall k\geqslant k_0, ~ \forall m\in \mathbb{N}.
\vspace{0.2cm}\\
(|u_{km}|u_{km}|v_{km}|)~\mbox{is bounded in } ~L^{\infty}(0, \infty; L^{2}(\Omega))), \quad \forall k\geqslant k_0, ~ \forall m\in \mathbb{N}.
\end{array}
\right.
\end{equation}

{\bf Second estimate.} Deriving $\eqref{EEE1}_1$ with respect to $t$, doing $w=u''_{km}(t)$ and as the function $F(\lambda)=|\lambda|, ~\lambda \in \mathbb{R}$ is Lipschitz continuous, using the \cite[Theorem A.3.12, Appendix--p.  35]{BC}  we obtain

\begin{equation}\label{se1}
\begin{split}
&\dfrac{1}{2} \dfrac{d}{dt}|u''_{km}(t)|^2 + \dfrac{1}{2} \dfrac{d}{dt}\|u'_{km}(t)\|^2 + \int_{\Gamma_1} \delta [u''_{km}(t)]^2d\Gamma \\
& \leqslant  \int_{\Omega}|u'_{km}(t)||u''_{km}(t)||v_{km}(t)|^2 dx 
+2 \int_{\Omega}|u_{km}(t)||u''_{km}(t)||v_{km}(t)||v'_{km}(t)| dx 
\end{split}
\end{equation} 
\\
Using the H\"{o}lder's inequality, the Sobolev embedding $V \hookrightarrow L^6(\Omega) $ and estimats \eqref{EEE10} we have
\begin{equation}\label{se2}
\begin{split}
\int_{\Omega}|u'_{km}(t)||u''_{km}(t)||v_{km}(t)|^{2} dx 
&\leqslant \|u'_{km}(t)\|_{L^6(\Omega)} |u''_{km}(t)| \|v_{km}(t)\|_{L^6(\Omega)}^2 \\
&  \leqslant C \|u'_{km}(t)\| |u''_{km}(t)|   \leqslant C(\|u'_{km}(t)\|^2 + |u''_{km}(t)|^2 ).
\end{split}
\end{equation} 
\\
Analogously we obtain
\begin{equation}\label{se3}
\int_{\Omega}|u_{km}(t)||u''_{km}(t)||v_{km}(t)| |v'_{km}(t)| dx \leqslant C(\|v'_{km}(t)\|^2 + |u''_{km}(t)|^2 ),
\end{equation} 	
where $C$ denote the several constants independent of $k$ and $m$.
Replecing \eqref{se2} and \eqref{se3} in \eqref{se1} and using the fact that $\delta(x) \geqslant \delta_0>0$ we have
\[
\dfrac{1}{2} \dfrac{d}{dt}|u''_{km}(t)|^2+ \dfrac{1}{2} \dfrac{d}{dt}\|u'_{km}(t)\|^2 +\delta_0 \int_{\Gamma_1}[u''_{km}(t)]^2d\Gamma \leqslant C(\|u'_{km}(t)\|^2 + \|v'_{km}(t)\|^2 + 2|u''_{km}(t)|^2).
\]
In similar way 
\[
\dfrac{1}{2} \dfrac{d}{dt}|v''_{km}(t)|^2+ + \dfrac{1}{2} \dfrac{d}{dt}\|v'_{km}(t)\|^2 +\delta_0 \int_{\Gamma_1}[v''_{km}(t)]^2d\Gamma \leqslant C(\|u'_{km}(t)\|^2 + \|v'_{km}(t)\|^2 + 2|v''_{km}(t)|^2).
\]
Adding the lest equalities above and integrating on $[0,t]$

\begin{equation}\label{se4}
\begin{split}
&\dfrac{1}{2}(|u''_{km}(t)|^2+|v''_{km}(t)|^2+ \|u'_{km}(t)\|^2 +\|v'_{km}(t)\|^2 ) + \delta_0 \int_{0}^{t}\int_{\Gamma_1}[u''_{km}(s)]^2d\Gamma ds \\
&+ \delta_0 \int_{0}^{t}\int_{\Gamma_1}[v''_{km}(s)]^2d\Gamma ds \leqslant \dfrac{1}{2}(|u''_{km}(0)|^2+|v''_{km}(0)|^2+ \|u_{k}^1\|^2 +\|v_{k}^1\|^2 ) \\
& + \int_{0}^{t} C(|u''_{km}(s)|^2+|v''_{km}(s)|^2+ \|u'_{km}(s)\|^2 +\|v'_{km}(s)\|^2 ) ds.
\end{split}
\end{equation}
We need to bound $|u''_{km}(0)|^2$ and $|v''_{km}(0)|^2$ by a constant independent of $k$ and $m$. This is one of the key points of the proof. These bounds are obtained thanks to the choice of the special basis of $V \cap H^2(\Omega)$. In fact, taking $t=0$ in $\eqref{EEE1}_1$ we obtain

\begin{equation}\label{boundL2}
(u''_{km}(0), w) + ((u_{km}(0),w)) + \int_{\Gamma_1}\delta(x)u'_{km}(0) w d\Gamma + \int_{\Omega}|u_{km}(0)||v_{km}(0)|v_{km}(0)w dx=0.
\end{equation}
We have $u_{km}(0)=u_{k}^0, ~ u'_{km}(0)=u_{k}^1$ and $\dfrac{\partial u_k^0}{\partial \nu}=-\delta(x)u_k^1$ on $\Gamma_1$. Aplying Green formula to \eqref{boundL2}, we then obtain

\[
(u''_{km}(0), w) = (\Delta u_k^0,w) + \int_{\Omega}|u_{k}^0||v_{k}^0|v_{k}^0w dx,
\] 
for all $w \in V_{m}^k$. Taking $w=u''_{km}(0)$ in this equality, using H\"{o}lder's inequality and observing the convergences \eqref{c1} we have
\[
|u''_{km}(0)| \leqslant |\Delta u_k^0| + \|u_k^0\|_{L^6(\Omega)}  \|v_k^0\|_{L^6(\Omega)}^2\leqslant C, \quad \forall k, m.
\]
Thus $(u''_{km}(0))$ is bounded in $L^2(\Omega)$, for all $k,m$. In similar way $(v''_{km}(0))$ is bounded in $L^2(\Omega)$, for all $k,m$.

From \eqref{se4}, observing the fact $(u''_{km}(0)), (v''_{km}(0))$ are bounded in $L^2(\Omega)$ and the convergences \eqref{c1} we have 

\[
\begin{split}
&\dfrac{1}{2}(|u''_{km}(t)|^2+|v''_{km}(t)|^2+ \|u'_{km}(t)\|^2 +\|v'_{km}(t)\|^2 ) + \delta_0 \int_{0}^{t}\int_{\Gamma_1}[u''_{km}(s)]^2d\Gamma ds \\
&+ \delta_0 \int_{0}^{t}\int_{\Gamma_1}[v''_{km}(s)]^2d\Gamma ds \leqslant C 
+ \int_{0}^{t} C(|u''_{km}(s)|^2+|v''_{km}(s)|^2+ \|u'_{km}(s)\|^2 +\|v'_{km}(s)\|^2 ) ds.
\end{split}
\]
Therefore by Gronwall's inequality there exists $C(t), ~t>0$ such that
\[
|u''_{km}(t)|^2+|v''_{km}(t)|^2+ \|u'_{km}(t)\|^2 +\|v'_{km}(t)\|^2 + \int_{0}^{t}\int_{\Gamma_1}[u''_{km}(s)]^2d\Gamma ds 
+ \int_{0}^{t}\int_{\Gamma_1}[v''_{km}(s)]^2d\Gamma ds \leqslant C(t), 
\]
it follows that 
\begin{equation}\label{se5}
\left|
\begin{array}{l}
(u'_{km}), (v'_{km})~\mbox{are bounded in } ~L_{loc}^{\infty}(0, \infty; V); \vspace{0.2cm}\\
(u''_{km}), (v''_{km})~\mbox{are bounded in } ~L_{loc}^{\infty}(0, \infty; L^2(\Omega));\vspace{0.2cm}\\
(u''_{km}), (v''_{km})~\mbox{are bounded in} ~L_{loc}^{2}(0, \infty; L^2(\Gamma_1)).
\end{array}
\right.
\end{equation}
{\bf Passage to the limit}.
Estimates \eqref{EEE10}, \eqref{se5} and using symilar arguments used in Theorem \ref{theorem1} with $\rho=1$ allow us
\begin{equation}\label{Z7}
\left|
\begin{array}{l}
u_{km} \to u_k \quad \mbox{and} \quad  v_{km} \to v_k \quad \mbox{weak star in} \quad  L^{\infty}(0,\infty, V), \vspace{0.3cm}\\
u'_{km} \to u'_k \quad \mbox{and} \quad  v'_{km} \to v'_k \quad \mbox{weak star in} \quad  L_{loc}^{\infty}(0,\infty, V), \vspace{0.3cm}\\
u''_{km} \to u''_k \quad \mbox{and} \quad  v''_{km} \to v''_k \quad \mbox{weak star in} \quad  L_{loc}^{\infty}(0,\infty, L^2(\Omega)), \vspace{0.3cm}\\
u'_{km} \to u'_k \quad \mbox{and} \quad  v'_{km} \to v'_k \quad \mbox{weak in} \quad  L^{2}(0,\infty, L^2(\Gamma_1)), \vspace{0.3cm}\\
u''_{km} \to u''_k \quad \mbox{and} \quad  v''_{km} \to v''_k \quad \mbox{weak in} \quad  L_{loc}^{2}(0,\infty, L^2(\Gamma_1)), \vspace{0.3cm}\\
|u_{km}||v_{km}|v_{km} \to |u_k||v_k|v_k \quad \mbox{weak star in} \quad  L^{\infty}(0,\infty, L^{2}(\Omega)),  \vspace{0.3cm}\\
|u_{km}|u_{km}|v_{km}| \to |u_k|u_k|v_k| \quad \mbox{weak star in} \quad  L^{\infty}(0,\infty, L^{2}(\Omega)). 
\end{array}
\right.
\end{equation}

From $\eqref{se5}_1$ and trace theorem we obtain
\[
(u'_{km}), (v'_{km})~\mbox{are bounded in } ~L_{loc}^{\infty}(0, \infty; H^{\frac{1}{2}}(\Gamma_1)).
\]
and thus
\begin{equation}\label{cf}
\left|
\begin{array}{l}
u'_{km} \to u_k' ~\mbox{weak star in } ~L_{loc}^{\infty}(0, \infty; H^{\frac{1}{2}}(\Gamma_1)); \vspace{0.2cm}\\
v'_{km} \to v_k ' ~\mbox{weak star in } ~L_{loc}^{\infty}(0, \infty; H^{\frac{1}{2}}(\Gamma_1)).
\end{array}
\right.
\end{equation}
Multiplying both sides of the approximate equation $\eqref{EEE1}_1$ by $\varphi \in \mathcal{D}(0,\infty)$,
integrating in $[0,\infty)$, using the convergences $\eqref{Z7}_{1,3,6}$, $\eqref{cf}_1$ we obtain 
\begin{equation}\label{aa}
\begin{split}
&\int_{0}^{\infty}(u_k''(t), w) \varphi(t) dt  +\int_{0}^{\infty} ((u_k(t), w)) \varphi(t) dt +\int_{0}^{\infty} \int_{\Gamma_1}\delta u_k'(t) w \varphi(t) d\Gamma dt 
\\ 
&+ \int_{0}^{\infty} ( |u_k(t)||v_k(t)|v_k(t),w ) \varphi(t) dt = 0,~ \forall  w \in V_m^k, ~\forall \varphi \in \mathcal{D}(0,T).
\end{split}
\end{equation}
Since $V_m^k$ is dense in $V\cap H^2(\Omega)$ it follows that \eqref{aa} is true for all $w \in V\cap H^2(\Omega)$.
In similar way, we find
\begin{equation*}
\begin{split}
&\int_{0}^{\infty}(v''_k(t), w)  \varphi(t) dt  +\int_{0}^{\infty} ((v_k(t), w)) \varphi(t) dt + \int_{0}^{\infty} \int_{\Gamma_1}\delta u'_k(t) w \varphi(t) d\Gamma dt
\\ 
&+ \int_{0}^{\infty} ( |u_k(t)|u_k(t)|v_k(t)|,w )\varphi(t) dt = 0, ~ \forall  w \in V \cap H^{2}(\Omega),~\forall \varphi \in \mathcal{D}(0,\infty).
\end{split}
\end{equation*}

We can see that the estimates \eqref{EEE10} and \eqref{se5} are also independent of $k$. Therefore by the same argument used to obtain \eqref{Z7} and \eqref{cf} we get a diagonal sequence $(u_{k_k}), (v_{k_k})$, still denoted by $(u_{k}), (v_{k_k})$, and functions $u,v: \Omega \times [0,\infty) \to \mathbb{R}$ such that
\begin{equation}\label{Z7c}
\left|
\begin{array}{l}
u_{k} \to u \quad \mbox{and} \quad  v_{k} \to v \quad \mbox{weak star in} \quad  L^{\infty}(0,\infty, V); \vspace{0.3cm}\\
u'_{k} \to u' \quad \mbox{and} \quad  v'_{k} \to v' \quad \mbox{weak star in} \quad  L_{loc}^{\infty}(0,\infty, V); \vspace{0.3cm}\\
u''_{k} \to u'' \quad \mbox{and} \quad  v''_{k} \to v'' \quad \mbox{weak star in} \quad  L_{loc}^{\infty}(0,\infty, L^2(\Omega)); \vspace{0.3cm}\\
u'_{k} \to u' \quad \mbox{and} \quad  v'_{k} \to v' \quad \mbox{weak in} \quad  L^{2}(0,\infty, L^2(\Gamma_1)); \vspace{0.3cm}\\
u''_{k} \to u'' \quad \mbox{and} \quad  v''_{k} \to v'' \quad \mbox{weak in} \quad  L_{loc}^{2}(0,\infty, L^2(\Gamma_1)); \vspace{0.3cm}\\
|u_{k}||v_{k}|v_{k} \to |u||v|v \quad \mbox{weak star in} \quad  L^{\infty}(0,\infty, L^{2}(\Omega));  \vspace{0.3cm}\\
|u_{km}|u_{km}|v_{km}| \to |u|u|v|  \quad \mbox{weak star in} \quad  L^{\infty}(0,\infty, L^{2}(\Omega));
\vspace{0.3cm}\\
u'_{k} \to u' \quad \mbox{and} \quad v'_{k} \to v' \quad \mbox{weak star in}  \quad L_{loc}^{\infty}(0, \infty; H^{\frac{1}{2}}(\Gamma_1)). 
\end{array}
\right.
\end{equation}
Taking the limit in \eqref{aa}, using convergences $\eqref{Z7c}_{1,3,6,8}$ and observing that $ V \cap H^{2}(\Omega)$ is dense in $V$, we obtain
\begin{equation}\label{i}
\begin{split}
&\int_{0}^{\infty}(u''(t), w) \varphi(t) dt  +\int_{0}^{\infty} ((u(t), w)) \varphi(t) dt +\int_{0}^{\infty} \int_{\Gamma_1}\delta u'(t) w \varphi(t) d\Gamma dt 
\\ 
&+ \int_{0}^{\infty} ( |u(t)||v(t)|v(t),w ) \varphi(t) dt = 0,~ \forall  w \in V, ~\forall \varphi \in \mathcal{D}(0,T).
\end{split}
\end{equation}
In similar way, we find
\begin{equation*}
\begin{split}
&\displaystyle\int_{0}^{\infty}(v''(t), z)\varphi(t) dt  +\int_{0}^{\infty} ((v(t), z)) \varphi(t) dt + \int_{0}^{\infty} \int_{\Gamma_1}\delta u'(t) z \varphi(t) d\Gamma dt
\\ 
&+ \displaystyle\int_{0}^{\infty} ( |u(t)|u(t)|v(t)|,z )\varphi(t)dt = 0, ~ \forall  z \in V,~\forall \varphi \in \mathcal{D}(0,T).
\end{split}
\end{equation*}

Taking in \eqref{i} $w \in \mathcal{D}(\Omega) \subset V$, it follows that
\[
u''-\Delta u + |u||v|v = 0 \quad \mbox{in} \quad \mathcal{D}'(\Omega).
\]
In similar way
\[
v''-\Delta v + |u|u|v|= 0 \quad \mbox{in} \quad \mathcal{D}'(\Omega).
\]
Therefore, by $\eqref{Z7}_{3,6}$, we get
\begin{equation}\label{G}
\begin{split}
u''-\Delta u + |u||v|v = 0 \quad \mbox{in} \quad L_{loc}^{\infty}(0,\infty, L^2(\Omega)) \\
\\
v''-\Delta v + |u|u|v|= 0 \quad \mbox{in} \quad L_{loc}^{\infty}(0,\infty, L^2(\Omega)).
\end{split}
\end{equation} 
As $u \in L^{\infty}(0,\infty;V)$ and by $\eqref{G}_1$, $\Delta u \in  L_{loc}^{\infty}(0,T;L^{2}(\Omega))$ then, by Milla Miranda \cite{MT} we obtain 
\begin{equation}\label{traco}
\frac{\partial u}{\partial \nu} \in  L^{\infty}(0,\infty;H^{-\frac{1}{2}}(\Gamma_1))
\end{equation}
Multiplyng both sides of \eqref{G} by $w\varphi$ with $w \in V$ and $\varphi \in \mathcal{D}(0,T)$, integrating on $\Omega \times (0,T)$ and using \eqref{traco} and Green's formula
\begin{equation*}
\begin{split}
&\displaystyle\int_{0}^{\infty}(u''(t), w) \varphi(t) dt  +\int_{0}^{\infty} ((u(t), w)) \varphi(t) dt - \int_{0}^{\infty} \left\langle \frac{\partial u(t)}{\partial \nu},w \right\rangle \varphi(t)d\Gamma dt
\\ 
&+ \displaystyle\int_{0}^{\infty} (|u(t)|u(t)|v(t)|,w )\varphi(t) dt = 0, ~ \forall  w \in V, ~\forall \varphi \in \mathcal{D}(0,T),
\end{split}
\end{equation*}
where $\langle \cdot,\cdot\rangle$ denotes the duality paring between $H^{-\frac{1}{2}}(\Gamma_1)$ and $H^{\frac{1}{2}}(\Gamma_1)$. Comparing this lest equation with \eqref{aa}, we obtain
\[
\int_{0}^{\infty}\left\langle \frac{\partial u(t)}{\partial \nu}+\delta u'(t),w \right\rangle \varphi(t) d\Gamma dt=0, ~ \forall  w \in V ~\forall \varphi \in \mathcal{D}(0,T)
\]
Therefore
\[
\frac{\partial u}{\partial \nu}+\delta u'=0 \quad \mbox{in} \quad H^{-\frac{1}{2}}(\Gamma_1).
\]
In similar way
\[
\frac{\partial v}{\partial \nu}+\delta v'=0 \quad \mbox{in} \quad H^{-\frac{1}{2}}(\Gamma_1).
\]
Snce $\delta u' \in L_{loc}^{\infty}(0,\infty, H^{\frac{1}{2}}(\Gamma_1))$, then
\begin{equation}\label{eult}
\frac{\partial u}{\partial \nu}+\delta u'=0 \quad \mbox{in} \quad L_{loc}^{\infty}(0,\infty, H^{\frac{1}{2}}(\Gamma_1)).
\end{equation}
In similar way
\begin{equation}\label{eult1}
\frac{\partial v}{\partial \nu}+\delta v'=0 \quad \mbox{in} \quad L_{loc}^{\infty}(0,\infty, H^{\frac{1}{2}}(\Gamma_1)).
\end{equation}
To complete the proof of the Theorem \ref{theorem3}, we shall show that $u \in L_{loc}^{\infty}(0,\infty, H^2(\Omega))$. In fact, note that $u \in L^{\infty}(0,\infty;V)$. This and $\eqref{G}_1$, \eqref{eult} imply that $u$ is the solution of the following boundary value problem:
\begin{equation*}
\left|
\begin{array}{l}
-\Delta u = f \quad \mbox{in} \quad  \Omega \times [0,T]; \vspace{0.2cm}\\
u = 0 ~\quad \mbox{on} \quad  \Gamma_0 \times [0,T]; \\
\frac{\partial u}{\partial \nu} = g \quad \mbox{on} \quad  \Gamma_1 \times [0,T],
\end{array}
\right.
\end{equation*}
for all real number $T>0$. Since 
\[
f=-u''-|u||v|v \in L_{loc}^{\infty}(0,\infty;L^2(\Omega)) \quad \mbox{and} \quad g=-\delta u' \in L_{loc}^{\infty}(0,\infty; H^{\frac{1}{2}}(\Gamma_1)), 
\]
it follows by Proposition \ref{propos} that 
\begin{equation}\label{reg}
u \in L_{loc}^{\infty}(0,\infty;V\cap H^2(\Omega)).
\end{equation}
In similar way 
\begin{equation}\label{reg1}
v \in L_{loc}^{\infty}(0,\infty;V\cap H^2(\Omega)).
\end{equation}
\\
{\bf Initial Conditions} The verification of the initial conditions follows by similar arguments used in the Theorem \ref{theorem1} in the previous section.
\\
{\bf Uniqueness}: The uniqueness results from the energy method. 

\begin{remark}
From \eqref{Z7}, we obtain $u_k$ and $v_k$ in the class \eqref{class3}. From \eqref{aa} and using the same arguments for obtain \eqref{G}, \eqref{eult} and \eqref{eult1}, we get 
\begin{equation}\label{agora}
	\begin{split}
	&u_k'' - \Delta u_k + |u_k||v_k|v_k =0 \quad \text{in} \quad  L^{\infty}_{loc}(0,\infty;L^{2}(\Omega))\\ 
	&v_k'' - \Delta v _k+ |u_k|u_k|v_k| =0 \quad \text{in} \quad L_{loc}^{\infty}(0,\infty; L^{2}(\Omega))\\
	&\frac{\partial u_k}{\partial \nu} + \delta(x)u_k'=0 \quad \mbox{in} \quad L^{\infty}_{loc}(0,\infty; H^{\frac{1}{2}}(\Gamma_1))	\\
	&\frac{\partial v_k}{\partial \nu} + \delta(x)v_k'=0 \quad \mbox{in} \quad  L^{\infty}_{loc}(0,\infty; H^{\frac{1}{2}}(\Gamma_1)).
	\end{split}
	\end{equation}
Also using the same arguments for obtain the regularities \eqref{reg} and \eqref{reg1}, we get
\begin{equation}\label{r}
u_k, v_k \in L_{loc}^{\infty}(0,\infty;V\cap H^2(\Omega)).
\end{equation}

\end{remark}
\begin{corollary}\label{corollary3}
We obtain similar results to Theorem \ref{theorem3} for the case $\rho>1$ and $n=1,2$.
\end{corollary}
\section{Assymptotic behavior}
In this section we prove an exponential decay result for the solution obtained in Theorem \ref{theorem3} and Corollary \ref{corollary3}, that is, $\rho=1$ if $n \leqslant 3$ and $\rho>1$ if $n=1,2$. 

We make the proof for $\rho=1$ and $n=3$. The result for $\rho=1$ and $n=1,2$ are derived in similar way.

To prove the Theorem \ref{teo2.2} we show that the energy
\begin{equation*}\label{Em}
\begin{split}
E_k(t)&=\frac{1}{2}\{|u'_{k}(t)|^2 +|v'_{k}(t)|^2 + \|u_{k}(t)\|^2  +  \|v_{k}(t)\|^2 \} \\
&+ \frac{1}{2}\int_{\Omega} (|u_{k}(t)|u_{k}(t))(|v_{k}(t)|v_{k}(t))dx,~ t \in [0, \infty).
\end{split}
\end{equation*}
associated with the solutions $\{u_k(t), v_k(t)\}$ of the equations in \eqref{agora} satisfies the inequality $\eqref{EE22}.$ Thus, the exponential decay of $E(t)$, will be obtained by taking the $\liminf$ of $E_k(t)$ as $k \to \infty$.

Now, we introduce the function
\begin{equation*}\label{AB1}
\begin{split}
\psi_k(t)&=2(u'_{k}(t),m \cdot \nabla u_{k}(t))+(n-1)(u'_{k}(t),u_{k}(t))\\
&+2(v'_{k}(t),m \cdot \nabla v_{k}(t))+(n-1)(v'_{k}(t),v_{k}(t)).
\end{split}
\end{equation*}  
For $\varepsilon >0,$ we introduce the perturbed energy
\begin{equation*}\label{AB2}
E_{k\varepsilon }(t) := E_{k}(t)  + \varepsilon \psi_k(t).
\end{equation*}

First we prove that $E_{k\varepsilon}(t)$ and $E_{k}(t)$ are equivalent. Then we show that
\begin{equation}\label{EN}
E'_{k\varepsilon}(t) \leqslant -\varepsilon E_{k}(t).
\end{equation}
\subsection{Equivalence between $E_{k\varepsilon}(t)$ and $E_{k}(t)$}
First of all, we note that
\begin{equation}\label{A1}
A_k(t)=\frac{1}{4}( \|u_{k}(t)\|^2  +  \|v_{k}(t)\|^2 ) + \frac{1}{2}\int_{\Omega} (|u_{k}(t)|u_{k}(t))(|v_{k}(t)|v_{k}(t))dx \geqslant 0, \quad \forall t\in[0,\infty).
\end{equation}
In fact,
\[
\frac{1}{2}\left|\int_{\Omega} (|u_{k}(t)|u_{k}(t))(|v_{k}(t)|v_{k}(t))dx\right| \leqslant \frac{1}{4}c_1^4 (\|u_{k}(t)\|^4  +  \|v_{k}(t)\|^4 ).
\]
Then 
\begin{equation}\label{A2}
A_k(t) \geqslant \frac{1}{4}\|u_{k}(t)\|^2 - \frac{1}{4}c_1^4\|u_{k}(t)\|^4 + \frac{1}{4}\|v_{k}(t)\|^2  - \frac{1}{4}c_1^4 \|v_{k}(t)\|^4.
\end{equation}
As $ - \frac{1}{4}c_1^4>-N_1$, we obtain
\[
\frac{1}{4}\|u_{k}(t)\|^2 - \frac{1}{4}c_1^4\|u_{k}(t)\|^4\geqslant \frac{1}{4}\|u_{k}(t)\|^2 - N_1\|u_{k}(t)\|^4 .
\]
If we take the limits $m \to \infty$ in Lemma \ref{L1}, we find
\begin{equation}\label{A3}
J_1(\|u_{k}(t)\|)=\frac{1}{4}\|u_{k}(t)\|^2 - N_1\|u_{k}(t)\|^4 \geqslant 0, \quad \forall t \in [0,\infty).
\end{equation}
In similar way
\begin{equation}\label{A4}
\frac{1}{4}\|v_{k}(t)\|^2 - N_1\|v_{k}(t)\|^4 \geqslant 0, \quad \forall t \in [0,\infty).
\end{equation}
Taking into account \eqref{A3} and \eqref{A4} in \eqref{A2}, we derive \eqref{A1}.

Observe that
\[
E_k(t)\geqslant \frac{1}{4}(|u'_{k}(t)|^2+|v'_{k}(t)|^2)+\frac{1}{4}(\|u_{k}(t)\|^2+\|v_{k}(t)\|^2)+A_k(t).
\]
Then by \eqref{A1}
\begin{equation}\label{A5}
E_k(t)\geqslant \frac{1}{4}(|u'_{k}(t)|^2+|v'_{k}(t)|^2)+\frac{1}{4}(\|u_{k}(t)\|^2+\|v_{k}(t)\|^2),\quad \forall t \in [0,\infty).
\end{equation}
On the other side, we have
\[
|\psi_k(t)| \leqslant \left(R+\frac{n-1}{2}\right)(|u_k'(t)|^2 + |v_k'(t)|^2) + \left(R+\frac{n-1}{2\lambda_1}\right)( \|u_k(t)\|^2 + \|v_k(t)\|^2),
\]
where $\lambda_1$ is the first eigenvalue of the spectral problem $((u,v))=\lambda(u,v), ~ u,v \in V$. Thus
\begin{equation}\label{A6}
|\psi_k(t)| \leqslant \frac{P}{4}(|u'_{k}(t)|^2+|v'_{k}(t)|^2+\|u_{k}(t)\|^2+\|v_{k}(t)\|^2), 
\end{equation}
where $P$ was defined in \eqref{const}.

From \eqref{A5} and \eqref{A6} it follows that
\[
|\psi_k(t)| \leqslant PE_k(t), \quad \forall t\in[0,\infty).
\]
Since that 
\[
|E_{k\varepsilon }(t) - E_k(t)| = \varepsilon|\psi_k(t)| \leqslant \varepsilon P E_k(t),
\]
we have
\[
E_k(t)(1- \varepsilon P) \leqslant E_{k\varepsilon}(t) \leqslant (1+ \varepsilon P) E_k(t).
\]
Then
\begin{equation}\label{AB5}
\frac{1}{2} E_k(t) \leqslant E_{k\varepsilon_1 }(t) \leqslant \frac{3}{2} E_k(t), \quad 0 < \varepsilon_1 \leqslant \frac{1}{2P}.
\end{equation}

From now on, to simplify the notation we will do not write the variable t.

\subsection{Proof of} \eqref{EN}.
Multiplying $\eqref{agora}_{1}$ and $\eqref{agora}_{2}$ by $u'_k$ and $v'_k$, respectively, using the fact $\delta(x)=m(x)\cdot \nu(x)$ the hypothesis \eqref{const}, we get
\begin{equation}\label{AB3}
E'_k(t)\leqslant -m_0(\|u'_k(t)\|_{L^2(\Gamma_1)}^2+\|v'_k(t)\|_{L^2(\Gamma_1)}^2).
\end{equation}

The idea to prove \eqref{EN} is to find that
\[
\begin{split}
\psi'_k(t) &\leqslant - E_{k}(t)- \left[\frac{1}{4}(\|u_k(t)\|^2+\|v_k(t)\|^2)-N_1(\|u_k(t)\|^4+\|v_k(t)\|^4)\right] \\
&+ D(\|u'_k(t)\|_{L^2(\Gamma_1)}^2+\|v'_k(t)\|_{L^2(\Gamma_1)}^2),
\end{split}
\]
where $N_1$ and $D$ are positive constants independent of $k$. 

Then to prove an existence theorem of solutions which permits us to say
\[
\frac{1}{4}(\|u_k(t)\|^2+\|v_k(t)\|^2)-N_1(\|u_k(t)\|^4+\|v_k(t)\|^4)\geqslant 0, \quad \forall t \in [0,\infty).
\]
Thus
\[
E'_{k\varepsilon}(t)= E'_{k}(t)  + \varepsilon \psi'_k(t)\leqslant - \varepsilon E_{k}(t)-(m_0 -\varepsilon D)(\|u'_k(t)\|_{L^2(\Gamma_1)}^2+\|v'_k(t)\|_{L^2(\Gamma_1)}^2).
\]
For small $\varepsilon>0$, we obtain \eqref{EN}.

Differentiating the function $\psi_k,$ we obtain
\[
\begin{split}
\psi_k' &=  2(u_k'',m \cdot \nabla u_k) + 2(u_k',m \cdot \nabla u_k')+ (n-1)(u_k'',u_k) + (n-1)|u_k'|^2 \\
&+ 2(v_k'',m \cdot \nabla v_k) + 2(v_k',m \cdot \nabla v_k')+ (n-1)(v_k'',v_k) + (n-1)|v_k'|^2 
\end{split}
\]
From $\eqref{agora}_1$ and $\eqref{agora}_2,$ we find
\begin{equation}\label{psilinha}
\begin{split}
\psi_k'&=2(\Delta u_k, m \cdot \nabla u_k) - 2(|u_k||v_k|v_k ,m \cdot \nabla u_k)+2(u_k',m \cdot \nabla u_k') + (n-1) (\Delta u_k ,u_k) \\
&- (n-1)(|u_k||v_k|v_k ,u_k) + (n-1)|u_k'|^2 \\
&+ 2(\Delta v_k, m \cdot \nabla v_k) - 2(|u_k|u_k|v_k|,m \cdot \nabla v_k)+2(v'_k,m \cdot \nabla v_k) + (n-1) (\Delta v_k ,v_k) \\
&- (n-1)(|u_k|u_k|v_k| ,v_k) + (n-1)|v'_k|^2 \\
& =: I_1+ \cdots + I_{12},
\end{split}
\end{equation}

respectively. 

Our goal is to derive a bound above for each terms on the right hand side of \eqref{psilinha}.

\begin{itemize}
	
	\item The regularity \eqref{r} allows us to obtain Rellich's identity for $u_k$, see \cite[Remark 2.3, p. 41]{Zuazua}, that is,
	\[
	I_1=2(\Delta u_k,m\cdot \nabla u_k)=(n-2)\|u_k\|^2+2\int_{\Gamma} \frac{\partial u_k}{\partial \nu} (m \cdot \nabla u_k)d\Gamma - 
	\int_{\Gamma}(m \cdot \nu)|\nabla u_k|^2 d\Gamma.
	\]
	Since $|\nabla u_k|^2=\left(\frac{\partial u_k}{\partial \nu}\right)^2$ and $m \cdot \nabla u_k=(m \cdot \nu)\frac{\partial u_k}{\partial 
		\nu}$ on $\Gamma_0$, then
	\[
	\int_{\Gamma} (m \cdot \nu)|\nabla u_k|^2 d\Gamma =  \int_{\Gamma_0} (m \cdot \nu)\left(\frac{\partial u_k}{\partial \nu}
	\right)^2d\Gamma+ \int_{\Gamma_1} (m \cdot \nu)|\nabla u_k|^2 d\Gamma
	\]
	and
	\[
	\int_{\Gamma} \frac{\partial u_k}{\partial \nu}(m \cdot \nabla u_k) d\Gamma =  \int_{\Gamma_0} (m \cdot \nu)\left(\frac{\partial u_k}
	{\partial \nu}\right)^2 d\Gamma+ \int_{\Gamma_1} \frac{\partial u_k}{\partial \nu}(m \cdot \nabla u_k) d\Gamma.
	\]
	Thus 
	\begin{equation}\label{decay1}
	\begin{split}
	I_1&=(n-2)\|u_k\|^2+ 2\int_{\Gamma_0} (m \cdot \nu)\left(\frac{\partial u_k}{\partial \nu}\right)^2 d\Gamma+ 2\int_{\Gamma_1} \frac{\partial u_k}{\partial \nu}(m \cdot \nabla u_k) d\Gamma \\
	&- \int_{\Gamma_0} (m \cdot \nu)\left(\frac{\partial u_k}{\partial \nu}\right)^2 d\Gamma - \int_{\Gamma_1} (m \cdot \nu)|\nabla u_k|^2 d\Gamma
	\end{split}
	\end{equation}
	As $\dfrac{\partial u_k}{\partial \nu} + (m \cdot \nu)u_k'=0$ on $\Gamma_1$ we have
	\begin{equation}\label{decay2}
	\begin{split}
	& \left|2 \int_{\Gamma_1} \frac{\partial u_k}{\partial \nu}(m \cdot \nabla u_k) d\Gamma \right|=  \left|2 \int_{\Gamma_1} (m \cdot \nu)u'_k(m \cdot \nabla u_k) d\Gamma \right| \\
	& \leqslant R^3 \int_{\Gamma_1}|u'_k|^2 d\Gamma + \int_{\Gamma_1}(m \cdot \nu)|\nabla u_k|^2.
	\end{split}
	\end{equation}
	Substituting \eqref{decay2} in \eqref{decay1}, making the reduction of similar terms and observing that $m \cdot \nu \leqslant 0$ on $\Gamma_0$, we get
	
	\begin{equation}\label{decay3}
	I_1 \leqslant (n-2)\|u_k\|^2 + R^3 \|u'_k\|_{L^2(\Gamma_1)}.
	\end{equation}
	In similar way
	\begin{equation}\label{decay31}
	I_7 \leqslant (n-2)\|v_k\|^2 + R^3\|v'_k\|_{L^2(\Gamma_1)}.
	\end{equation}
	
	\item Note that 
	\[
	\begin{split}
	I_2=-2(|u_k||v_k|v_k, m\cdot \nabla u_k)=-\sum_{i=1}^{n}\int_{\Omega}m_i\left[\frac{\partial }{\partial x_i}(|u_k|u_k)\right](|v_k|v_k)dx.
	\end{split}
	\]
	Similarly we find
	\[
	I_8=-2(|u_k|u_k|v_k|, m\cdot \nabla v_k)=-\sum_{i=1}^{n}\int_{\Omega}m_i(|u_k|u_k)\left[\frac{\partial }{\partial x_i}(|v_k|v_k)\right]dx.
	\]
	Thus from Green's theorem and using the fact $\dfrac{\partial m_i}{\partial x_i}=1$ and $u_k=v_k=0$ on $\Gamma_0$, we have
	\[
	I_2+I_8=n\int_{\Omega}(|u_k|u_k)(|v_k|v_k)dx- \int_{\Gamma_1}(m\cdot \nu)(|u_k|u_k)(|v_k|v_k)d\Gamma.
	\]
	Using \eqref{gamma}, we have
	\[
	\left|-\int_{\Gamma_1}(m\cdot \nu)(|u_k|u_k)(|v_k|v_k)d\Gamma \right| \leqslant \frac{Rc_2^{4}}{2}\left(\|u_k\|^{4}+\|v_k\|^{4} \right).
	\] 
	Therefore
	
	\begin{equation}\label{decay4}
	I_2+I_8 \leqslant n\int_{\Omega}(|u_k|u_k)(|v_k|v_k)dx+ \frac{Rc_2^{4}}{2}\left(\|u_k\|^{4}+\|v_k\|^{4} \right)
	\end{equation}
	
	\item From Green's theorem and since $\dfrac{\partial m_i}{\partial x_i}=1$ and $u'_k=0$ on $\Gamma_0$, we obtain
\begin{equation}\label{decay5}
I_3=2(u'_k,m \cdot \nabla u'_k) \leqslant -n |u'_k|^2+R\|u'_k\|_{L^2(\Gamma_1)}^2.
\end{equation}
	In Similar way
	\begin{equation}\label{decay51}
	I_9 \leqslant -n |v'_k|^2+R\|v'_k\|_{L^2(\Gamma_1)}^2.
	\end{equation}
	
	\item From the boundary conditions $\dfrac{\partial u_k}{\partial \nu}+(m \cdot \nu)u'_k=0$ on $\Gamma_1$, we find
	\[
	(\Delta u_k, u_k)=-\|u_k\|^2-\int_{\Gamma_1}(m\cdot \nu)u_k'u_k d\Gamma 
	\]
	Note that, using \eqref{gamma} we obtain
	\[
	\begin{split}
	\left|\int_{\Gamma_1}(m\cdot \nu)u_k'u_k d\Gamma \right| & \leqslant R \int_{\Gamma_1}|u_k'||u_k| d\Gamma \\
	&\leqslant \frac{1}{2}R^2c_3^2(n-1)\|u_k'\|_{L^2(\Gamma_1)}^2 + \frac{1}{2(n-1)c_3^2}\|u_k\|_{L^2(\Gamma_1)}^2 \\
	&\leqslant \frac{1}{2}R^2(n-1)c_3^2\|u_k'\|_{L^2(\Gamma_1)}^2 + \frac{1}{2(n-1)}\|u_k\|^2.
	\end{split}
	\]
	Thus
	\begin{equation}\label{decay6}
	I_4 =(n-1)(\Delta u_k,u_k)\leqslant -(n-1)\|u_k\|^2 + \frac{1}{2}R^2(n-1)^2c_3^2\|u_k'\|_{L^2(\Gamma_1)}^2 + \frac{1}{2}\|u_k\|^2.
	\end{equation}
	In similar way 
	\begin{equation}\label{decay61}
	I_{10} \leqslant -(n-1)\|v_k\|^2 + \frac{1}{2}R^2(n-1)^2c_3^2\|v_k'\|_{L^2(\Gamma_1)}^2 + \frac{1}{2}\|v_k\|^2.
	\end{equation}
	
	\item From \eqref{H1} we have
	\[
	(|u_k||v_k|v_k,u_k)\leqslant \int_{\Omega}|u_k|^{2}|v_k|^{2}dx \leqslant \frac{c_1^{4}}{2}(\|u_k\|^{4} + \|v_k\|^{4}).
	\]
	Therefore
	\begin{equation}\label{decay7}
	I_5=-(n-1)(|u_k||v_k|v_k,u_k) \leqslant \frac{c_1^{4}}{2}(n-1)(\|u_k\|^{4} + \|v_k\|^{4}).
	\end{equation}
	In similar way 
	\begin{equation}\label{decay71}
	I_{11} \leqslant \frac{c_1^{4}}{2}(n-1)(\|u_k\|^{4} + \|v_k\|^{4}).
	\end{equation}
\end{itemize}
Taking into account \eqref{decay3}--\eqref{decay71} in \eqref{psilinha} and reducing similar terms, we obtain
\[
\begin{split}
I_1+ \cdots+I_{12} & \leqslant -(|u_k'|^2+|v_k'|^2)-\frac{1}{2} (\|u_k\|^2+\|v_k\|^2)+n\int_{\Omega}(|u_k|u_k)(|v_k|v_k)dx \\
&+\left[\frac{Rc_2^4}{2}+c_1^4(n-1)\right](\|u_k\|^{4} + \|v_k\|^{4})+ D( \|u_k'\|_{L^2(\Gamma_1)}^2 + \|v_k'\|_{L^2(\Gamma_1)}^2)
\end{split}
\]
where $D$ defined in \eqref{const}. Thus,
\begin{equation}\label{decay712}
\begin{split}
I_1+ \cdots+I_{12} & \leqslant -\frac{1}{2}(|u_k'|^2+|v_k'|^2)-\frac{1}{4} (\|u_k\|^2+\|v_k\|^2)-\frac{1}{4}\int_{\Omega}(|u_k|u_k)(|v_k|v_k)dx\\
&-\frac{1}{4}(\|u_k\|^2+\|v_k\|^2)+\left[n+\frac{1}{4}\right]\int_{\Omega}(|u_k|u_k)(|v_k|v_k)dx \\
&+\left[\frac{Rc_2^4}{2}+c_1^4(n-1)\right](\|u_k\|^{4} + \|v_k\|^{4})+ D( \|u_k'\|_{L^2(\Gamma_1)}^2 + \|v_k'\|_{L^2(\Gamma_1)}^2)
\end{split}
\end{equation}
Now using \eqref{H1}, we have
\begin{equation}\label{1s1}
\left|\int_{\Omega}(|u_k|v_k)(|v_k|v_k)dx\right| \leqslant \frac{c_1^{4}}{2}(\|u_k\|^{4} + \|v_k\|^{4}).
\end{equation}
Combining \eqref{1s1} with \eqref{decay712} we get
\[
\begin{split}
I_1+ \cdots+I_{12} & \leqslant -\frac{1}{2}E_k-\left(\frac{1}{4}(\|u_k\|^2+\|v_k\|^2)-N_1(\|u_k\|^{4} + \|v_k\|^{4})\right) + D( \|u_k'\|_{L^2(\Gamma_1)}^2 + \|v_k'\|_{L^2(\Gamma_1)}^2),
\end{split}
\]
where $N_1$ was defined in \eqref{N}. From \eqref{A3} and \eqref{A4} we obtain
\[
\frac{1}{4}(\|u_k\|^2+\|v_k\|^2)-N_1(\|u_k\|^{4} + \|v_k\|^{4}) \geqslant 0 .
\]
Therefore
\begin{equation}\label{decay10}
\psi_k' \leqslant - \frac{1}{2} E_k + D( \|u_k'\|_{L^2(\Gamma_1)}^2 + \|v_k'\|_{L^2(\Gamma_1)}^2).
\end{equation}
From \eqref{AB3} and \eqref{decay10} and knowing that $E'_{k\varepsilon}=E_k'+\varepsilon\psi_k'$ we have
\[
E'_{k\varepsilon} \leqslant -\frac{\varepsilon}{2} E_k-(m_0-D\varepsilon)(\|u_k'\|_{L^2(\Gamma_1)}^2 + \|v_k'\|_{L^2(\Gamma_1)}^2).
\]
Therefore 
\begin{equation}\label{decay11}
E'_{k\varepsilon_2} \leqslant -\frac{\varepsilon_2}{2} E_k,\quad \mbox{for all} \quad 0\leqslant \varepsilon_2 \leqslant \frac{m_0}{D}.
\end{equation}

\subsection{Proof of Theorem \ref{teo2.2}}
The choice $\tau$ given in \eqref{const} implies that \eqref{AB5} and \eqref{decay11} hold simultaneously for this $\tau$. Thus, from \eqref{AB5} we have
\[
-\frac{\tau}{2} E_k \leqslant -\frac{\tau}{3} E_{k\tau}.
\]
Consequently, using the above inequality in $\eqref{decay11},$ we obtain
\[
E'_{k\tau}  \leqslant -\frac{\tau}{3} E_{k\tau}.
\]
This give us that 
\[
E_{k\tau}(t)  \leqslant e^{-\frac{\tau}{3} t} E_{k\tau}(0).
\]
From this inequality and $\eqref{AB5}$ we have 
\[
E_k(t)  \leqslant 3 E_k(0)e^{-\frac{\tau}{3} t}, \quad \mbox{for all} \quad t \in [0,\infty).
\]
Taking the $\liminf$ of $E_k(t)$ in the above inequality we conclude the proff of Theorem \ref{teo2.2}.



\begin{thebibliography}{9}

\bibitem{2015} A. B. Aliev and A. A. Kazimov; {\it Nonexistence of global solutions of the Cauchy problem for systems of Klein-Gordon equations with positive initial energy differential equations}, Differentsial'nye Uravneniya, 2015, Vol. 51, No. 12, p. 1587-1592.

\bibitem{AMZ}A. F. Almeida, M. M. Cavalcanti and J. P. Zanchetta; {\it Exponential decay for the coupled Klein-Gordon-Schr\"odinger equations with locally distributed damping}, Communications on Pure and Applied Analysis, 2018, Vol. 17, No. 5, p. 2039-2061.

\bibitem{Cavalcanti} V. Bisognin, M.M. Cavalcanti, V.N. Domingos Cavalcanti and J. Soriano; {\it Uniform decay for the coupled Klein-Gordon-Schr\"odinger equations with locally distributed damping}, Nonlinear Differential Equations Appl. 2008, 15, p. 91-113.

\bibitem{BC} H. Brezis and T. Cazenave; {\it Nonlinear Evolution Equations}, IM-UFRJ, Rio, 1994.

\bibitem{2} H. ChunYan, G. BoLing, H. DaiWen and L. Q. Xin; {\it Global well-posedness of the fractional Klein-Gordon-Schr\"odinger system with rough initial data}, Science China Mathematics, Vol. 59, No. 7, p.1345-1366 

\bibitem{2018} M. Dolgopolik,  A. L. Fradkov and B. Andrievsky; {\it Boundary energy control of a system governed by the nonlinear Klein-Gordon equation}, Math. Control Signals Syst. (2018), 30:7.

\bibitem{FM} J. Ferreira and G. Perla Menzala; {\it Decay of solutions of a system of nonlinear Klein-Gordon equations}.

\bibitem{2014} D. Kim and H. Sunagawa; {\it Remarks on decay of small solutions to systems of Klein-Gordon equations with dissipative nonlinearities}, Nonlinear Analysis 97, (2014), p. 94-105.

\bibitem{Komornik} V. Komornik; {\it Exact Controllability  and Stabilization. The Multiplier Method}, John Wiley and Sons-Masson, Paris, 1994.

\bibitem{Zuazua} V. Komornik and E. Zuazua; {\it A Direct Method for the Boundary Stabilization of the Wave Equation}, J. Math. pures and at appl., 69, 1990, p. 33-54.

\bibitem{2016} M. O. Korpusov and S. G. Mikhailenko; {\it Blow-up of the Solution to the Cauchy Problem with Arbitrary Positive Energy for a System of Klein-Gordon Equations in the de Sitter Metric}, Computational Mathematics and Mathematical Physics, 2016, Vol. 56, No. 10, p. 1758-1762.

\bibitem{Liu} W. Liu;  {\it Global existence, asymptotic behavior and blow-up of solutions for coupled Klein-Gordon equations with damping terms}, Nonlinear Analysis, 73, (2010), p. 244-255

\bibitem{A1}A. T. Louredo and M. Milla Miranda; {\it Nonlinear Boundary Dissipation for a Coupled System of Klein-Gordon Equations}, Electronic Journal of Differential Equations, Vol. 2010(2010), No. 120, p. 1-19.

\bibitem{MT} Milla Miranda; {\it Tra\c co para o dual dos espaços de Sobolev}, Bol. Soc. Paran. Matemática ($2\textordfeminine$ série) 11(2) (1990), p. 131-157.

\bibitem{MM} L. A. Medeiros and M. Milla Miranda; {\it Espa\c cos de Sobolev : Inicia\c c\~ao aos problemas el\'iticos n\~ao
homog\^eneos} - Rio de Janeiro: IM-UFRJ, 2010,185 p.	

\bibitem{MM} L.A. Medeiros and G. Perla Menzala; {\it On a Mixed for a Class of Nonlinear Klein-Gordon Equations}, Acta Math. Hung. 52(1-2)(1988), p. 61-69, Rio de Janeiro.

\bibitem{Carrier}  M. Milla Miranda, A.T. Louredo and L.A. Medeiros; {\it On nonlinear wave equations of Carrier type}, Journal of Mathematical Analysis and Applications, 432 (2015), p. 565-582.

\bibitem{MA1}   M. Milla Miranda and L. A. Medeiros; {\it Weak Solutions for a System of Nonlinear Klein-Gordon Equations}, Institute of Mathematics - UFRJ, 1985. 

\bibitem{MA2}  M. Milla Miranda and L. A. Medeiros; {\it On the Existence of Global Solutions of a Coupled Nonlinear Klein-Gordon Equations}, Funkcialaj Ekvacioj, 30 (1987), p. 147-161.

\bibitem{MMLA1996} M. Milla Miranda and L. A. Medeiros; {\it On a boundary value problem for wave equations: existence, uniqueness-asymptotic behavior}, Rev. Mat. Apl., 17 (1996), p. 47-73. MR1413482.

\bibitem{necas} J. Ne\v{c}as; {\it Direct Methods in the Theory of Elliptic Equations}, (Springer Monographs in Mathematics), Springer Heidelberg Dordrecht London New York.

\bibitem{Park} J.Y. Park and J.U. Jeong; {\it Optimal control of damped Klein-Gordon equations with state constraints}, J. Math. Anal. Appl. 334 (2007), p. 11-27.

\bibitem{Sattinger} D. H. Sattinger; {\it On global solutions of nonlinear hyperbolic equations}, Arch. Rational Mech. Anal., 30 (1968), p. 148-172.

\bibitem{Segal} I. Segal; {\it Nonlinear partial differential equations in quantum field theory}, Proc. Symp. Appl. Math. AMS, 17(1965), p. 210-226.

\bibitem{LT} L. Tartar; {\it Topics in Nonlinear Analysis}, Uni. Paris Sud, Dep. Math., Orsay, France, (1978).

\bibitem{RT} R. Temam; {\it Navier-Stokes Equations (Studies in mathematics and its applications, v. 2)}, Elsevier North-Holland, Inc. New York, 1979.
	
\bibitem{Wang}Y.J. Wang; {\it Non-existence of global solutions of a class of coupled nonlinear Klein-Gordon equations with non-negative potentials and arbitrary initial energy}, IMA J. Appl. Math. 74 (3) (2009), p. 392-415.

\end{thebibliography}
\end{document}